\documentclass[preprint,12pt,3p]{elsarticle}
\usepackage{amsmath,amsthm,amsfonts,amssymb}
\usepackage{hyperref}
\usepackage{yhmath}
\usepackage{MnSymbol}  
\usepackage[normalem]{ulem}
\usepackage{sectsty}

\usepackage{colortbl}
\usepackage[dvipsnames]{xcolor}
\newtheorem{theorem}{Theorem}[section]
\newtheorem{lemma}[theorem]{Lemma}

\newtheorem{example}[theorem]{Example}

\newtheorem{remark}[theorem]{Remark}
\newtheorem{thm}{Theorem}[section]

\newtheorem{cor}[thm]{Corollary}

\newtheorem{prop}[thm]{Proposition}

\newtheorem{rem}[thm]{Remark}

\numberwithin{equation}{section}

\definecolor{darkslategray}{rgb}{0.18, 0.31, 0.31}
\definecolor{warmblack}{rgb}{0.0, 0.26, 0.26}
\definecolor{astral}{RGB}{46,116,181}
\subsectionfont{\color{astral}}
\sectionfont{\color{astral}}
\journal{....................... }

\begin{document}
	\begin{frontmatter}
		\title{ \textcolor{warmblack}{\bf New $\mathbb{A}$-numerical radius equalities and inequalities for certain operator matrices and applications}}
		\author[label1]{Soumitra Daptari}\ead{daptarisoumitra@gmail.com }
			\author[label2]{Fuad Kittaneh\corref{cor2}}\ead{fkitt@ju.edu.jo}
			
			\address[label2]{Department of Mathematics, The University of Jordan, Amman, Jordan}
		
		\author[label1]{Satyajit Sahoo}\ead{ssahoo@niser.ac.in, satyajitsahoo2010@gmail.com}

		\address[label1]{School of Mathematical Sciences, National Institute of Science Education and Research Bhubaneswar, India}

		\cortext[cor2]{Corresponding author}

		\begin{abstract}
			\textcolor{warmblack}{
			The main goal of this article is to establish  several new  $\mathbb{A}$-numerical radius equalities and inequalities for $n\times n$ cross-diagonal, left circulant, skew left circulant operator matrices, where $\mathbb{A}$ is the $n\times n$ diagonal operator matrix whose diagonal entries are positive bounded operator $A$. Also, we introduce two new matrices called left imaginary circulant operator matrix and left imaginary skew circulant operator matrix and present their $\mathbb{A}$-numerical radii. Certain $\mathbb{A}$-numerical radii of general $n\times n$ operator matrices are obtained.  Some special cases of our results lead to the results of earlier works in the literature, which shows that our results are more general. Applications of our results are established through some interesting examples. We also provide a concluding section by posing a problem for future research.   }
		\end{abstract}
		
	\end{frontmatter}

	\section{Introduction and Preliminaries }\label{intro}
	Let $\mathcal{H}$ be a complex Hilbert space with inner product $\langle \cdot,\cdot\rangle$  and the corresponding norm $\|\cdot\|$.
	Let  $\mathcal{B}(\mathcal{H})$ be the $C^*$-algebra of all bounded linear operators on $\mathcal{H}$. We let $\mathbb{H}=\displaystyle\bigoplus_{i=1}^n\mathcal{H}$ be the direct sum of $n$ copies of $\mathcal{H}$. If $T_{ij}, 1\leq i, j\leq n$ are operators in $\mathcal{B(H)}$, then operator matrix $\mathbb{T}=[T_{i,j}]$ can be defined on $\mathbb{H}$ by $$ \mathbb{T}x=\begin{bmatrix}
	\displaystyle\sum_{j=1}^{n}T_{1j}x_j\\
	\vdots\\
	\displaystyle\sum_{j=1}^{n}T_{nj}x_j
	\end{bmatrix}$$ for every vector $x=[x_1,\dots, x_n]^T\in \mathbb{H}$. If $S_i\in \mathcal{B(H)},  i=1,\dots,n$, we denote their direct sum by $\displaystyle\bigoplus_{i=1}^n S_i$, which is the $n\times n$ block diagonal operator matrix,
	
	$$ \displaystyle\bigoplus_{i=1}^nS_i =\begin{bmatrix}
	S_1 & &&\\
	& S_2&&\\
	&& \ddots\\
	&&& S_n
	\end{bmatrix}.$$

If $T_i\in \mathcal{B(H)}, i=1,\dots,n$, then the circulant operator matrix $\mathbb{T}_{circ}=\mbox{circ}(T_1,\dots,T_n)$ is the $n\times n$ matrix whose first row has entries $T_1, \dots, T_n$ and the other rows are obtained by successive cyclic permutations of these entries. The left-circulant operator matrix $\mathbb{T}_{lcirc}=\mbox{lcirc}(T_1,\dots,T_n)$ is the $n\times n$ matrix whose first row has entries $T_1, \dots, T_n$ and the other rows are obtained by successive {\it left shift} cyclic permutations of these entries. It is well-known that $\mbox{lcirc}(T_1,\dots,T_n)=\mbox{circ}(T_n,\dots,T_1)J,$ where $J=$off-diag$(I, \dots, I).$ 
 
 The skew circulant operator matrix $\mathbb{T}_{scirc}=\mbox{scirc}(T_1,\dots,T_n)$ is the $n\times n$ circulant followed by a change in sign to all the elements below the main diagonal. It is well-known that every skew circulant operator matrix is unitarily equivalent to a circulant operator matrix. For more details, one can visit the excellent book \cite{Davis}. The skew left circulant operator matrix $\mathbb{T}_{slcirc}=\mbox{slcirc}(T_1,\dots,T_n)$ is the $n\times n$ left circulant followed by a change in sign to all the elements under the off-diagonal. Clearly, $ \mbox{slcirc}(T_1,\dots,T_n)=\mbox{scirc}(T_n,\dots,T_1)J$. Note that circulant matrices with complex entries are normal and hence their norms and numerical radii are equal. But, left circulant matrices are not normal in general.\\
Motivated by the notions of imaginary circulant operator matrix and imaginary skew circulant operator matrix in \cite{KITSAT}, we have defined the following two operator matrix, say left imaginary circulant operator matrix and left imaginary skew circulant operator matrix.
 
	If $T_i\in \mathcal{B(H)}, i=1,\dots,n$, then the imaginary left  circulant operator matrix $\mathbb{T}_{lcirc_i}=\mbox{lcirc}_i(T_1,\dots,T_n)$ is the $n\times n$ matrix whose first row has entries $T_1, \dots, T_n$ and the other rows are obtained by successive cyclic permutations of $i$-multiplies of these entries, i.e.,
	$\mathbb{T}_{lcirc_i}=\begin{bmatrix}
	T_{1} & T_{2}  & \cdots  &T_{n-1}  & T_{n}\\
	T_{2}  & T_{3} & \cdots  & T_{n-1}&iT_1\\
	T_{3}  & T_{4} & \ddots  &  iT_{1}&iT_2\\
	\vdots & \vdots & \ddots & \vdots&\vdots \\
	T_{n}  & iT_{1}  &\cdots& i T_{n-2} & iT_{n-1}
	\end{bmatrix}$.
	The imaginary skew left circulant operator matrix $\mathbb{T}_{lscirc_i}=\mbox{lscirc}_i(T_1,\dots,T_n)$ is the $n\times n$ imaginary left circulant followed by a change in sign to all the elements below the off-diagonal. Thus,
	$\mathbb{T}_{slcirc_i}=\begin{bmatrix}
	T_{1} & T_{2}  & \cdots  &T_{n-1}  & T_{n}\\
	T_{2}  & T_{3} & \cdots  & T_{n-1}&-iT_1\\
	T_{3}  & T_{4} & \ddots  &  -iT_{1}&-iT_2\\
	\vdots & \vdots & \ddots & \vdots&\vdots \\
	T_{n}  & -iT_{1}  &\cdots& -i T_{n-2} & -iT_{n-1}
	\end{bmatrix}$.
	
	Let $\mathcal{L}$ be subspace of $\mathcal{H}$, we denote use the symbol  $\overline{\mathcal{L}}$
	to denote the norm closure of $\mathcal{L}$ in the norm topology of
	$\mathcal{H}$. We also denote the orthogonal projection onto a closed linear subspace $\mathcal{M}$ of $\mathcal{H}$
	by $P_{\mathcal{M}}$.  The identity operator and the null operator are designated with the notations $I$ and $O$, respectively on $\mathcal{H}$. Let $A\in\mathcal{B}(\mathcal{H})$, we use the notations $\mathcal{R}(A)$, $\mathcal{N}(A)$, $A^{*}$ for range, null space, and 
	adjoint of $A$, respectively. 
	An  operator $A\in\mathcal{B}(\mathcal{H})$  is called {\it positive} if $\langle Ax, x\rangle \geq 0$ for all $x \in \mathcal{H}$, and is called {\it strictly positive} if  $\langle Ax, x\rangle > 0$ for all non-zero $x\in \mathcal{H}$. We denote a positive (strictly positive) operator $A$ by $A \geq 0$ ($A > 0$).  Throughout this paper, we assume that $A \in\mathcal{B}(\mathcal{H})$ is a positive operator, and $\mathbb{A} \in\mathcal{B}(\displaystyle\bigoplus_{i=1}^n\mathcal{H})$ is an $n\times n$ diagonal operator matrix whose diagonal entries are the positive operator $A$. Then, this operator $A$ defines a positive semidefinite sesquilinear form: $$\langle \cdot,\cdot\rangle_A: \mathcal{H}\times\mathcal{H}\rightarrow\mathbb{C}, ~~ \langle x, y \rangle_A=\langle Ax,y\rangle, ~~x,y\in\mathcal{H}.$$ 
	Let $\|\cdot\|_A$ denote the seminorm on $\mathcal{H}$ induced by $\langle \cdot, \cdot \rangle_A,$ i.e. $\|x\|_A=\sqrt{\langle x, x \rangle_A}$ for all $x \in \mathcal{H}.$ 
	Note that $\|x\|_A=0$ if and only if $x\in \mathcal{N}(A)$.
	And, $\|x\|_A$ is a norm
	if and only if $A$ is one-to-one (or $A>0$). Also, $(\mathcal{H}, \|\cdot\|_A)$ is complete if and only if  $\mathcal{R}(A)$ is closed in $\mathcal{H}.$ \\

	The $A$-operator seminorm of $T \in\mathcal{B}\mathcal{(H)}$, denoted by $\|T\|_A$'
	is defined as follows:
	$$\|T\|_A=\sup_{x\in \overline{\mathcal{R}(A)},~x\neq 0}\frac{\|Tx\|_A}{\|x\|_A}<\infty.$$
	There is another equivalent definition of $\|T\|_A$, see \cite{Zam}.
	Let $\mathcal{B}^A\mathcal{(H)}$ signify the set of all bounded linear operators on $\mathcal{H}$ whose $A$-operator seminorm is finite.
	It can be shown that $\mathcal{B}^A\mathcal{(H)}$ is not a subalgebra of $\mathcal{B(H)}$, and $\|T\|_A=0$ if and only if $T^*AT=O.$ For $T\in\mathcal{B}^A\mathcal{(H)},$ we have 
	$$\|T\|_A=\sup \{|\langle Tx,y\rangle_A|: x,y\in \overline{\mathcal{R}(A),} ~\|x\|_A=\|y\|_A=1\}.$$
	If $AT\geq 0$, then the operator $T$ is called {\it $A$-positive}.  Note that if $T$ is $A$-positive, then 
	$$\|T\|_A=\sup \{\langle Tx,x\rangle_A: x\in \mathcal{H},  \|x\|_A=1\}.$$
	An operator $X\in \mathcal{B(H)}$ is called an {\it $A$-adjoint operator} of $T\in \mathcal{B(H)}$ if  $\langle Tx,y\rangle_A=\langle x, Xy\rangle_A$  for every $x,y\in \mathcal{H},$ i.e., if $AX=T^*A.$
	By \cite{Doug, Manu}, the existence of an $A$-adjoint operator is not guaranteed.  An operator $T\in \mathcal{B(H)}$ may admit none, one or many $A$-adjoints. An $A$-adjoint  of an operator $T\in \mathcal{B}\mathcal{(H)}$ exists if and only if $\mathcal{R}(T^*A)\subseteq \mathcal{R}(A)$. The set of all operators which admit $A$-adjoints is denoted by  $\mathcal{B}_A\mathcal{(H)}$. 
	Note that $\mathcal{B}_A\mathcal{(H)}$ is a subalgebra of $\mathcal{B(H)}$, which is neither closed nor dense in $\mathcal{B(H)}.$ Moreover, the following inclusions $\mathcal{B}_A\mathcal{(H)}\subseteq \mathcal{B}^A\mathcal{(H)}\subseteq\mathcal{B}\mathcal{(H)}$ hold with equality if $A$ is injective and has a closed range.\\
	
	If $T\in \mathcal{B}_A\mathcal{(H)},$ the reduced solution of the equation $AX=T^*A$ is a distinguished $A$-adjoint operator of $T,$ which is denoted by $T^{\#_A}$ (see \cite{ARIS2,Mos}). Note that $T^{\#_A}=A^\dagger T^* A$, where $A^\dagger$ is the Moore–Penrose inverse  \cite{Nashed} of $A$.   Recall that $A^\dagger:R(A)\bigoplus R(A)^{\perp} \longrightarrow \mathcal{H}$ is the unique 
	operator satisfying
	$AA^\dagger A = A$,~  $A^\dagger A A^\dagger = A^\dagger$,~$A^\dagger A= P_{N(A)^{\perp}}$,~ $A A^\dagger=P_{\overline{\mathcal{R}(A)}}|_{R(A)\bigoplus R(A)^{\perp}}$.
	If $T\in \mathcal{B}_A(\mathcal{H}),$ then $AT^{\#_A}=T^*A$, $\mathcal{R}(T^{\#_A})\subseteq \overline{\mathcal{R}(A)}$ and $\mathcal{N}(T^{\#_A})=\mathcal{N}(T^*A)$ (see \cite{Doug}). 
	An operator $T\in \mathcal{B(H)}$ is said to be {\it $A$-selfadjoint} if $AT$ is selfadjoint, i.e., if $AT=T^*A.$ Observe that if $T$ is $A$-selfadjoint, then $T\in \mathcal{B}_A(\mathcal{H}).$  The next example shows that $T\neq T^{\#_A}.$ 
	Let $A=\begin{bmatrix}
	4&2\\2&1
	\end{bmatrix}$ and $T=\begin{bmatrix}
	2&1\\4&2
	\end{bmatrix}.$ Then $T$ is $A$-selfadjoint, but $T^{\#_A}=\begin{bmatrix}
	16/5&8/5\\8/5&4/5
	\end{bmatrix}\neq T.$  
	For $T\in\mathcal{B}_A(\mathcal{H}),$ $T=T^{\#_A}$   if and only if $T$ is $A$-selfadjoint and $\mathcal{R}(T)\subseteq \overline{\mathcal{R}(A)}.$ If $T\in \mathcal{B}_A(\mathcal{H}),$ then $T^{\#_A}\in \mathcal{B}_A(\mathcal{H}),$  $(T^{\#_A})^{\#_A}=P_{\overline{\mathcal{R}(A)}}TP_{\overline{\mathcal{R}(A)}},$  and $\left((T^{\#_A})^{\#_A}\right)^{\#_A}=T^{\#_A}.$ Also, $T^{\#_A}T$ and $TT^{\#_A}$ are  $A$-positive operators, and
	\begin{align}\label{ineq0}
	\|T^{\#_A}T\|_A=\|TT^{\#_A}\|_A=\|T\|_A^2=\|T^{\#_A}\|_A^2.
	\end{align}
	An operator $U\in \mathcal{B}_A(\mathcal{H})$ is said to be {\it $A$-unitary} if $\|Ux\|_A=\|U^{\#_A}x\|_A=\|x\|_A$ for all $x\in \mathcal{H}.$ 
	For $T,S\in \mathcal{B}_A(\mathcal{H}),$ we have $(TS)^{\#_A}=S^{\#_A}T^{\#_A},$  $(T+S)^{\#_A}=T^{\#_A}+S^{\#_A},$ $\|TS\|_A\leq \|T\|_A\|S\|_A$ and $\|Tx\|_A\leq \|T\|_A\|x\|_A$ for all $x\in \mathcal{H}.$  The real and imaginary parts of an operator $T\in \mathcal{B}_A(\mathcal{H})$  are $Re_A(T)=\frac{T+T^{\#_A}}{2}$ and $Im_A(T)=\frac{T-T^{\#_A}}{2i}$.  
	The interested reader may refer to \cite{ARIS2} for further properties of operators on semi-Hilbert spaces.\\
	
	The {\it numerical radius} of $T\in
	\mathcal{B(H)}$ is denoted by $w(T)$,  which is
	defined as $$ w(T)=\displaystyle\sup\{|\langle Tx, x \rangle|: x\in \mathcal{H}, \|x\|=1 \}.$$
	It is well-known that
	$w(\cdot)$ defines a norm on $\mathcal{B}(\mathcal{H})$, and is equivalent to the
	usual operator norm $\|T\|=\displaystyle \sup  \{ \|Tx \|: x\in \mathcal{H}, \|x\|=1 \}.$ In fact, for
	every $T \in \mathcal{B(H)}$, 
	\begin{align}\label{p3100}
	\frac{1}{2}\|T\|\leq w(T)\leq \|T\|.
	\end{align}
 The first inequality becomes an equality if $T^2=0$. The second inequality becomes an equality if $T$ is normal.
	There is a plethora of literature (for example, \cite{MBKS,HirKit,TY,MOS_SAT,SND,SND1,SND2}) that studies different generalizations, refinements and applications of numerical radius inequalities, to which we refer the interested reader. In 2012, Saddi \cite{Saddi} introduced the {\it $A$-numerical radius} of $T$ for $T\in \mathcal{B(H)}$, which is denoted as $w_A(T)$, and is defined as follows:
	\begin{equation*}\label{eqn1a}
	w_A(T)=\sup\{|\langle Tx,x\rangle_A|:x\in \mathcal{H}, \|x\|_A=1\}.
	\end{equation*}
	It then follows that
	\begin{equation}\label{eqn_00000001.4}
	w_A(T)=w_A(T^{\#_A}) \mbox{  for   any  }   T\in\mathcal{B}_A(\mathcal{H}). 
	\end{equation}
	If $T\in \mathcal{B}_A(\mathcal{H})$ and $U$ is $A$-unitary, then $w_A(U^{\#_A}TU)=w_A(T).$ In 2019, Zamani \cite{Zam} came up with the following new formula for computing the numerical radius of $T\in \mathcal{B}_A\mathcal{(H)}$:
	\begin{align*}
	w_A(T)=\sup_{\theta\in \mathbb{R}}\left\|\frac{e^{i\theta}T+(e^{i\theta}T)^{\#_A}}{2}\right\|_A.
	\end{align*}
	He then extended the renowned inequality \eqref{p3100} using the $A$-numerical radius of $T$, and the same is produced below:
	\begin{align}\label{ineq1}
	\frac{1}{2}\|T\|_A\leq w_A(T)\leq \|T\|_A.
	\end{align}
 Moreover, if $T^2=0$ then
\begin{align}\label{eq1.5}
    w_A(T)=\frac{1}{2}\|T\|_A.
\end{align}
	Furthermore,  if $T$ is $A$-selfadjoint, then $w_A(T)=\|T\|_A$.
	In 2019,  Moslehian {\it et al.} \cite{MOS}  pursued the study of the $A$-numerical radius and established some $A$-numerical radius inequalities.
	In 2020, Bhunia {\it et al.}  \cite{PINTU} obtained several $\mathbb{A}$-numerical radius inequalities. 
	Further generalizations and refinements of $A$-numerical radius inequalities are discussed in \cite{Pintu1, Pintu2, Feki, {NSD}}. For more results on $\mathbb{A}$-numerical radius inequalities, we refer the reader to \cite{Feki,Feki01,faiot,feki03,Nirmal2,XYZ,SS, KITSAT}. In 2020, the concept of the $A$-spectral radius of $A$-bounded operators was introduced by Feki in \cite{Feki01} as follows:
		$r_A(T):=\displaystyle\inf_{n\geq 1}\|T^n\|_A^{\frac{1}{n}}=\displaystyle\lim_{n\to\infty}\|T^n\|_A^{\frac{1}{n}}.$ 
  The main inspiration of our results comes from the second and third author's earlier work \cite{KITSAT}.
 Kittaneh et al. \cite{KITSAT} established $\mathbb{A}$-numerical radius equalities for $n\times n$ circulant, skew circulant, imaginary circulant, imaginary skew circulant, tridiagonal, and anti-tridiagonal operator matrices. The results of the current article is a continuation of earlier work of a recent article by Kittaneh et al.\cite{KITSAT}.
 
In this aspect, the rest of the paper is organized as follows. Motivated by the work of  Kittaneh and Sahoo \cite{KITSAT}, we establish certain $\mathbb{A}$-numerical radius equalities for $n\times n$ cross-diagonal operator matrices in Section 2. In Section 3, we obtain $\mathbb{A}$-numerical radius equalities and inequalities for $n\times n$ left circulant, skew left circulant operator matrices. Some special cases of our results have been given in this section. $\mathbb{A}$-numerical radii of some particular cases of our newely defined matrices i.e.,  left imaginary circulant operator matrix and left imaginary skew circulant operator matrix are consider in Section 4. Also, we present new upper bounds for the $\mathbb{A}$-numerical radii of general $n\times n$ operator matrices. Some applications of our results are established in Section 5. Finally, we end up with a conclusion section, which may produce new problems for future investigation.

	We need the following lemmas to prove our results. The following lemma is already proved by Bhunia et al. \cite{PINTU} for the case strictly positive operator $A$. Very recently, the same result was proved by Rout et al. \cite{Nirmal2} by dropping the assumption that $A$ is strictly positive, which is stated next for our purpose.
	
	\begin{lemma}\label{lem0001}\textnormal{[Lemma 2.4, \cite{Nirmal2}]} 
		Let $T_1, T_2\in {\mathcal{B}}_A(\mathcal{H}).$ Then 
		\begin{enumerate}
			\item [\textnormal{(i)}]
			$w_{\mathbb{A}}\left(\begin{bmatrix}
			T_1 & O\\
			O & T_2
			\end{bmatrix}\right)= \max\{w_A(T_1), w_A(T_2)\}.$\\
			\item [\textnormal{(ii)}] $w_{\mathbb{A}}\left(\begin{bmatrix}
			O & T_1\\
			T_2 & O
			\end{bmatrix}\right)=w_{\mathbb{A}}\left(\begin{bmatrix}
			O & T_2\\
			T_1 & O
			\end{bmatrix}\right).$\\
			\item [\textnormal{(iii)}] $w_{\mathbb{A}}\left(\begin{bmatrix}
			O & T_1\\
			e^{i\theta}T_2 & O
			\end{bmatrix}\right)=w_{\mathbb{A}}\left(\begin{bmatrix}
			O & T_1\\
			T_2 & O
			\end{bmatrix}\right)$ for~any~$\theta\in\mathbb{R}$.\\
			\item [\textnormal{(iv)}]  $w_{\mathbb{A}}\left(\begin{bmatrix}
			T_1 & T_2\\
			T_2 & T_1
			\end{bmatrix}\right)=\max\{w_A(T_1+T_2),w_A(T_1-T_2)\}.$
			In particular, $w_{\mathbb{A}}\left(\begin{bmatrix}
			O & T_1\\
			T_1 & O
			\end{bmatrix}\right)=w_A(T_1).$
		\end{enumerate}
	\end{lemma}
	\begin{lemma}\label{l002}\textnormal{[Lemma 2.9, \cite{Nirmal2}]} 
		Let $T_1, T_2 \in \mathcal{B}_A(\mathcal{H}).$ Then
		$$w_{\mathbb{A}}\left(\begin{bmatrix}
		T_2 & -T_1\\
		T_1 & T_2
		\end{bmatrix}\right)= \max\{w_A(T_1+iT_2), w_A(T_1-iT_2)\}.$$
	\end{lemma}

	\begin{lemma}\label{lemma1}\textnormal{[Lemma 3.1, \cite{pinfek}]}
		Let $T_{ij}\in{\mathcal{B}}_{A}(\mathcal{H}), 1\leq i,j\leq n$. Then\\ $${\begin{bmatrix}
			T_{11} & T_{12}  & \cdots  &   T_{1n}\\
			T_{21}  & T_{22} & \cdots  & T_{2n}\\
			\vdots & \vdots &\ddots & \vdots  \\
			T_{n1}  & T_{n2}  &\cdots&   T_{nn}
			\end{bmatrix}}^{\#_{\mathbb{A}}} =\begin{bmatrix}
		T_{11}^{\#_A} & T_{21}^{\#_A}  & \cdots  &   T_{n1}^{\#_A}\\
		T_{12}^{\#_A}  & T_{22}^{\#_A} & \cdots  & T_{n2}^{\#_A}\\
		\vdots & \vdots &\ddots & \vdots  \\
		T_{1n}^{\#_A}  & T_{2n}^{\#_A}  &\cdots&   T_{nn}^{\#_A}
		\end{bmatrix}.$$
		
	\end{lemma}
	\begin{lemma}\label{them3.7}\textnormal{[Theorem 3.5, \cite{NSD}]}
		Let $ T_i \in \mathcal{B}_A(\mathcal{H}),1\leq i\leq n$. Then 
		$$w_{\mathbb{A}}\left(\begin{bmatrix}
		T_1 &&\cdots&O\\
		O &T_2&&O\\
		\vdots &&\ddots &\vdots\\
		O&&\cdots&T_n
		\end{bmatrix}\right)=\max\{w_A(T_1),\dots, w_A(T_n)\}.$$
	\end{lemma}

\begin{lemma}\textnormal{(Theorem 3.4, \cite{NSD})}\label{thm3.4}
   Let $ A_i \in \mathcal{B}_A(\mathcal{H}),~ i= 1, 2,\ldots, n $ and $\mathbb{T}=\begin{bmatrix}
 O & \cdots  & O  & A_1 \\
 \vdots  &   & A_2 & O \\
 O & \adots &  & \vdots \\
 A_n  & O & \cdots &  O
 \end{bmatrix}.$ 
 If $n$ is even, then
 $$w_{\mathbb{A}}(\mathbb{T})\leq \frac{1}{2} \displaystyle\sum_{i=1}^{n} \|A_i \|_A,$$
 and if $n$ is odd, then
 $$ w_{\mathbb{A}}(\mathbb{T})\leq  w_A\bigg(A_{\frac{n+1}{2}}\bigg)+\frac{1}{2} \displaystyle\sum_{\substack{i=1\\{i\neq{\frac{n+1}{2}} }}}^{n} \|A_i\|_A.$$ 
\end{lemma}
\begin{lemma}\label{l87}\textnormal{(Proposition 1, \cite{Feki01})}\label{LEMA1.87}
Let $ S \in \mathcal{B}_A(\mathcal{H})$. Then $\rho_A(S^k)=(\rho_A(S))^k$ for all $k=1, 2, \dots$.
\end{lemma}
\begin{lemma}\label{l17}\textnormal{(Theorem 2.3, \cite{Feki})}\label{LEMA1.7}
    If $S=[S_{ij}]$ is an $n\times n$ operator matrix with $S_{ij} \in \mathcal{B}_A(\mathcal{H})$, then 
\begin{align}\label{p4106}
w_{\mathbb{A}}(S) \leq w([s_{ij}]), 
\end{align}
where 
\begin{align*}
s_{ij} =\left\{ \begin{array}{lll}
w_A(S_{ij}) & \textnormal{for $ i=j$}\\
w_{\mathbb{A}}\left(\begin{bmatrix}
O & S_{ij}\\
S_{ji} & O
\end{bmatrix}\right)& \textnormal{for $ i \neq j $}. 
\end{array} \right.
\end{align*}
\end{lemma}

	\section{$\mathbb{A}$-numerical radii of $n\times n$ cross-diagonal operator matrices}
	The aim of this section is to discuss certain $\mathbb{A}$-numerical radius equalities for $n\times n$ cross-diagonal operator matrices. The very first result is a formula for the  $\mathbb{A}$-numerical radius of a cross-diagonal operator matrix which is inspired by \cite{AKR}.

	\begin{theorem}\label{Thm1}
		Let $T_i\in \mathcal{B}_A(\mathcal{H}_i, \mathcal{H}_i)$ and  $S_i\in \mathcal{B}_A(\mathcal{H}_{n+1-i}, \mathcal{H}_i)$  for $1\leq i\leq n$ and \\let $\mathbb{T}=\begin{bmatrix}
		T_{1} & O &  & \cdots&  & O&  S_{1}\\
		O  & T_{2} & & & & S_{2}  & O\\
		\vdots  &  &\ddots &&\adots  & & \vdots\\
  	  &  &\adots& &\ddots  & &  \\
		O &S_{n-1} &  & && T_{n-1} & O \\
		S_{n}  & O  & &\cdots&  &O & T_{n}
		\end{bmatrix}$ be a cross-diagonal operator matrix. We assume that $T_{\frac{n+1}{2}}=S_{\frac{n+1}{2}}$, when $n$ is an odd number.  
  \begin{enumerate}
     \item [(i)] If $n$ is even, then
          \begin{align}\label{}
	w_\mathbb{A}(\mathbb{T})=	\max_{1\leq i\leq \frac{n}{2}}\left\{w_\mathbb{A}\left(\begin{bmatrix}
		T_{n-(2i-1)} & S_{n-(2i-1)}\\
            S_{2i} & T_{2i}
		\end{bmatrix}\right)\right\}.
      \end{align}
        \item [(ii)] If $n$ and $\frac{n+1}{2}$ are odd, then
      {  \scriptsize     \begin{align}\label{}
	w_\mathbb{A}(\mathbb{T})=\max\left\{	\max_{\substack{1\leq i\leq \frac{n-1}{4}}}\left\{w_\mathbb{A}\left(\begin{bmatrix}
		T_{n-(2i-1)} & S_{n-(2i-1)}\\
		S_{2i} & T_{2i}
		\end{bmatrix}\right)\right\}, w_{A}(T_{\frac{n+1}{2}}),\max_{\substack{\frac{n+7}{4}\leq i\leq \frac{n+1}{2}}}\left\{w_\mathbb{A}\left(\begin{bmatrix}
		T_{n-(2i-2)} & S_{n-(2i-2)}\\
		S_{2i-1} & T_{2i-1}
		\end{bmatrix}\right) \right\}\right\}.
      \end{align}}
      \item [(iii)] If $n$ is odd but $\frac{n+1}{2}$ is even, then
      {  \scriptsize     \begin{align}\label{}
	w_\mathbb{A}(\mathbb{T})=\max\left\{	\max_{\substack{1\leq i\leq \frac{n-3}{4}}}\left\{w_\mathbb{A}\left(\begin{bmatrix}
		T_{n-(2i-1)} & S_{n-(2i-1)}\\
		S_{2i} & T_{2i}
		\end{bmatrix}\right)\right\}, w_{A}(T_{\frac{n+1}{2}}),\max_{\substack{\frac{n+5}{4}\leq i\leq \frac{n+1}{2}}}\left\{w_\mathbb{A}\left(\begin{bmatrix}
		T_{n-(2i-2)} & S_{n-(2i-2)}\\
		S_{2i-1} & T_{2i-1}
		\end{bmatrix}\right) \right\}\right\}.
      \end{align}}
  \end{enumerate}
	\end{theorem}
 \begin{proof}
  Let us recall the following permutation $\pi$ for the set $\{1,2,\dots,n\}$ from \cite{AKR}. 

  If $n$ is even, then \begin{gather} 
		\pi(i)=
		\begin{cases}
		i; \text{ if $i$ even}\\
		n-i;\text{ if $i$ odd.}
		\end{cases}
		\end{gather}

   If $n$ is odd, then \begin{gather} 
		\pi(i)=
		\begin{cases}
		i; \text{ if $i$ even and $i<\frac{n+1}{2}$}\\
		n-i;\text{ if $i$ odd and $i<\frac{n+1}{2}$}\\
            n-i;\text{ if $i$ even and $i\geq \frac{n+1}{2}$}\\
            i; \text{ if $i$ odd and $i\geq\frac{n+1}{2}$}.
		\end{cases}
		\end{gather}
  We consider the matrix $\mathbb U=(U_{ij})$, where \begin{gather} 
		U_{ij}=
		\begin{cases}
		I; \text{ if } \pi(i)=j\\
		0;\text{ if }\pi(i)\neq j.
		\end{cases}
		\end{gather}
  Clearly, $\mathbb U^*=(U_{ji})$.

One can check that $\mathbb{U}^{\#_{\mathbb{A}}}=\mathbb A^{\dagger}\mathbb U^*\mathbb A=(U^{\#_{\mathbb{A}}}_{ij})$, where \begin{gather} 
		U^{\#_{\mathbb{A}}}_{ij}=
		\begin{cases}
		P_{\overline{\mathcal{R}(A)}}; \text{ if } \pi(j)=i\\
		0;\text{ if }\pi(j)\neq i.
		\end{cases}
		\end{gather}
  Now, $(\mathbb U \mathbb U^{\#_{\mathbb{A}}})_{ij}=\sum_{k=1}^{n}U_{ik}U^{\#_{\mathbb{A}}}_{kj}$. Clearly, if $i=j$, then $(\mathbb U \mathbb U^{\#_{\mathbb{A}}})_{ij}=U_{i\pi(i)}U^{\#_{\mathbb{A}}}_{\pi(i)i}=P_{\overline{\mathcal{R}(A)}}$, and if $i\neq j$, then $(\mathbb U \mathbb U^{\#_{\mathbb{A}}})_{ij}=U_{i\pi(i)}U^{\#_{\mathbb{A}}}_{\pi(i)j}+U_{i\pi(j)}U^{\#_{\mathbb{A}}}_{\pi(j)j}=U^{\#_{\mathbb{A}}}_{\pi(i)j}+U_{i\pi(j)}P_{\overline{\mathcal{R}(A)}}=0$. Therefore, \begin{gather} 
		(\mathbb U \mathbb U^{\#_{\mathbb{A}}})_{ij}=
		\begin{cases}
		P_{\overline{\mathcal{R}(A)}}; \text{ if } i=j\\
		0;\text{ if } i\neq j.
		\end{cases}
		\end{gather}
  Similarly, one can check that \begin{gather} 
		( \mathbb U^{\#_{\mathbb{A}}}\mathbb U)_{ij}=
		\begin{cases}
		P_{\overline{\mathcal{R}(A)}}; \text{ if } i=j\\
		0;\text{ if } i\neq j.
		\end{cases}
		\end{gather}
Thus, $\mathbb U \mathbb U^{\#_{\mathbb{A}}}=\mathbb U^{\#_{\mathbb{A}}}\mathbb U=\begin{bmatrix}
		P_{\overline{\mathcal{R}(A)}} & O  & \cdots  &   O\\
		O  & P_{\overline{\mathcal{R}(A)}} & \cdots  & O\\
		\vdots & \vdots &\vdots & \vdots \\
		O  & O  &\cdots&  P_{\overline{\mathcal{R}(A)}}
		\end{bmatrix}$. Hence, $\mathbb{U}$ is an $\mathbb{A}$-unitary operator. 

  Now, using Lemma \ref{lemma1} when $n$ is even, we have
  \begin{align*}
      \mathbb{U}^{\#_{\mathbb{A}}} \mathbb{T}^{\#_{\mathbb{A}}}  \mathbb{U} 
		&= \displaystyle\bigoplus_{i=1}^{\frac{n}{2}}\begin{bmatrix}
		P_{\overline{\mathcal{R}(A)}}T_{n-(2i-1)}^{\#_{A}} &P_{\overline{\mathcal{R}(A)}} S_{2i}^{\#_{A}}\\
		P_{\overline{\mathcal{R}(A)}}S_{n-(2i-1)}^{\#_{A}} & P_{\overline{\mathcal{R}(A)}}T_{2i}^{\#_{A}}
		\end{bmatrix}\\
  &= \displaystyle\bigoplus_{i=1}^{\frac{n}{2}}\begin{bmatrix}
		T_{n-(2i-1)}^{\#_{A}} &S_{2i}^{\#_{A}}\\
		S_{n-(2i-1)}^{\#_{A}} & T_{2i}^{\#_{A}}
		\end{bmatrix} ~~~~\textnormal{as}~  \mathcal{R}(T_i^{\#_A})\subseteq \overline{\mathcal{R}(A)}\\
  &= \left(\displaystyle\bigoplus_{i=1}^{\frac{n}{2}}\begin{bmatrix}
		T_{n-(2i-1)} &S_{n-(2i-1)}\\
	 S_{2i}	 & T_{2i}
		\end{bmatrix}\right)^{\#_{\mathbb A}}.
  \end{align*}

Similarly, if $n$ and $\frac{n+1}{2}$ are odd, we can show that 
 \begin{align*}
      \mathbb{U}^{\#_{\mathbb{A}}} \mathbb{T}^{\#_{\mathbb{A}}}  \mathbb{U} 
		&= \displaystyle\left\{\bigoplus_{i=1}^{\frac{n-1}{4}}\left(\begin{bmatrix}
		T_{n-(2i-1)}^{\#_{\mathbb{A}}} &S_{2i}^{\#_{\mathbb{A}}}\\
	 S_{n-(2i-1)}^{\#_{\mathbb{A}}}	 & T_{2i}^{\#_{\mathbb{A}}}
		\end{bmatrix}\right)\right\}\bigoplus T_{\frac{n+1}{2}}^{\#_{\mathbb{A}}} \bigoplus \left\{\bigoplus_{i=\frac{n+7}{4}}^{\frac{n+1}{2}}\left(\begin{bmatrix}
		T_{n-(2i-2)}^{\#_{\mathbb{A}}} &S_{2i-1}^{\#_{\mathbb{A}}}\\
	 S_{n-(2i-2)}^{\#_{\mathbb{A}}}	 & T_{2i-1}^{\#_{\mathbb{A}}}
		\end{bmatrix}\right)\right\}\\
  &= \displaystyle\left\{\bigoplus_{i=1}^{\frac{n-1}{4}}\left(\begin{bmatrix}
		T_{n-(2i-1)} &S_{n-(2i-1)}\\
	 S_{2i}	 & T_{2i}
		\end{bmatrix}\right)^{\#_{\mathbb A}}\right\}\bigoplus T^{\#_{\mathbb A}}_{\frac{n+1}{2}} \bigoplus \left\{\bigoplus_{i=\frac{n+7}{4}}^{\frac{n+1}{2}}\left(\begin{bmatrix}
		T_{n-(2i-2)} &S_{n-(2i-2)}\\
	 S_{2i-1}	 & T_{2i-1}
		\end{bmatrix}\right)^{\#_{\mathbb A}}\right\},
  \end{align*}
and if $n$ is odd but $\frac{n+1}{2}$ is even, then \begin{align*}
      \mathbb{U}^{\#_{\mathbb{A}}} \mathbb{T}^{\#_{\mathbb{A}}}  \mathbb{U} 
		&= \displaystyle\left\{\bigoplus_{i=1}^{\frac{n-3}{4}}\left(\begin{bmatrix}
		T_{n-(2i-1)}^{\#_{\mathbb{A}}} &S_{2i}^{\#_{\mathbb{A}}}\\
	 S_{n-(2i-1)}^{\#_{\mathbb{A}}}	 & T_{2i}^{\#_{\mathbb{A}}}
		\end{bmatrix}\right)\right\}\bigoplus T_{\frac{n+1}{2}}^{\#_{\mathbb{A}}} \bigoplus \left\{\bigoplus_{i=\frac{n+5}{4}}^{\frac{n+1}{2}}\left(\begin{bmatrix}
		T_{n-(2i-2)}^{\#_{\mathbb{A}}} &S_{2i-1}^{\#_{\mathbb{A}}}\\
	 S_{n-(2i-2)}^{\#_{\mathbb{A}}}	 & T_{2i-1}^{\#_{\mathbb{A}}}
		\end{bmatrix}\right)\right\}\\
  &= \displaystyle\left\{\bigoplus_{i=1}^{\frac{n-3}{4}}\left(\begin{bmatrix}
		T_{n-(2i-1)} &S_{n-(2i-1)}\\
	 S_{2i}	 & T_{2i}
		\end{bmatrix}\right)^{\#_{\mathbb A}}\right\}\bigoplus T^{\#_{\mathbb A}}_{\frac{n+1}{2}} \bigoplus \left\{\bigoplus_{i=\frac{n+5}{4}}^{\frac{n+1}{2}}\left(\begin{bmatrix}
		T_{n-(2i-2)} &S_{n-(2i-2)}\\
	 S_{2i-1}	 & T_{2i-1}
		\end{bmatrix}\right)^{\#_{\mathbb A}}\right\}.
  \end{align*}
  Now, using the fact that $w_{\mathbb{A}}(\mathbb{T})=w_{\mathbb{A}}(\mathbb{U}^{\#_{\mathbb{A}}}\mathbb{TU})$ for any $\mathbb{T}\in\mathcal{B}_A(\mathcal{H}),$ for $n$ is even, 
		we get
		
		\begin{align*}
		w_{\mathbb{A}}(\mathbb{T})=w_{\mathbb{A}}(\mathbb{T}^{\#_{\mathbb{A}}})=w_{\mathbb{A}}( \mathbb{U}^{\#_{\mathbb{A}}} \mathbb{T}^{\#_{\mathbb{A}}}  \mathbb{U} )&=w_{\mathbb{A}}\left(\left(\displaystyle\bigoplus_{i=1}^{\frac{n}{2}}\begin{bmatrix}
		T_{n-(2i-1)} &S_{n-(2i-1)}\\
	 S_{2i}	 & T_{2i}
		\end{bmatrix}\right)^{\#_{\mathbb A}}\right)\\
		&==w_{\mathbb{A}}\left(\left(\displaystyle\bigoplus_{i=1}^{\frac{n}{2}}\begin{bmatrix}
		T_{n-(2i-1)} &S_{n-(2i-1)}\\
	 S_{2i}	 & T_{2i}
		\end{bmatrix}\right)\right)\\
		&=\max_{1\leq i\leq \frac{n}{2}}\left\{w_\mathbb{A}\left(\begin{bmatrix}
		T_{n-(2i-1)} & S_{n-(2i-1)}\\
            S_{2i} & T_{2i}
		\end{bmatrix}\right)\right\},
				\end{align*}
		where the last equality follows from Lemma \ref{them3.7}.
	
  Using a similar argument for other $n$, we can find the $\mathbb A$-numerical radius.   

 \end{proof}
 \begin{rem}\label{Remark1}
 \begin{enumerate}
     \item[(i)]  If $T_i=O, i=1, \dots, n; i\neq \frac{n+1}{2}$  and $T_{\frac{n+1}{2}}=S_{\frac{n+1}{2}}$, then $$w_{\mathbb{A}}(\mathbb{T})=\displaystyle\max_{1\leq i\leq n}\left\{w_\mathbb{A}\left(\begin{bmatrix}
		O & S_{n-(2i-1)}\\
            S_{2i} & O
		\end{bmatrix}\right)\right\}=\displaystyle\max_{i+j=\frac{n+1}{2}}\left\{w_\mathbb{A}\left(\begin{bmatrix}
		O & S_{2j}\\
            S_{2i} & O
		\end{bmatrix}\right)\right\}.$$
     \item[(ii)] If $S_i=O, i=1, \dots, n; i\neq \frac{n+1}{2}$  and $T_{\frac{n+1}{2}}=S_{\frac{n+1}{2}}$, then $w_{\mathbb{A}}(\mathbb{T})=\displaystyle\max_{1\leq i\leq n}\{w_A(T_i)\}$.
 \end{enumerate}
  \end{rem}
 
  \begin{cor}
  Let $\mathbb T$ be the cross-diagonal matrix as defined in Theorem \ref{Thm1} with $T_{n-(2i-1)}=T_{2i}$ and $S_{n-(2i-1)}=S_{2i}, i=1,\dots, n$. Then 
   \begin{enumerate}
     \item [(i)] If $n$ is even, then
          \begin{align*}\label{}
	w_\mathbb{A}(\mathbb{T})=	\max_{1\leq i\leq \frac{n}{2}}\left\{w_A(T_{2i}+S_{2i}),w_A(T_{2i}-S_{2i}) \right\}.
      \end{align*}
        \item [(ii)] If $n$ and $\frac{n+1}{2}$ are odd, then the assumptions $T_{n-(2i-2)}=T_{2i-1}$ for $1\leq i \leq \frac{n-1}{4}$ and $S_{n-(2i-2)}=S_{2i-1}$ for $\frac{n+7}{4}\leq i \leq \frac{n+1}{2}$ yield that
   {  \scriptsize    \begin{align*}\label{}
	w_\mathbb{A}(\mathbb{T})=	\max\left\{\max_{\substack{1\leq i\leq \frac{n-1}{4}}}\left\{w_A(T_{2i}+S_{2i}),w_A(T_{2i}-S_{2i})\right\},w_A(T_{\frac{n+1}{2}}) ,\max_{\substack{\frac{n+7}{4}\leq i\leq \frac{n+1}{2}}}\left\{w_A(T_{2i-1}+S_{2i-1}),w_A(T_{2i-1}-S_{2i-1})\right\} \right\}.
      \end{align*}}
       \item [(iii)]  If $n$ is odd but $\frac{n+1}{2}$ is even,  then the assumptions $T_{n-(2i-2)}=T_{2i-1}$ for $1\leq i \leq \frac{n-3}{4}$ and $S_{n-(2i-2)}=S_{2i-1}$ for $\frac{n+5}{4}\leq i\leq  \frac{n+1}{2}$ yield that
      {  \scriptsize    \begin{align*}\label{}
	w_\mathbb{A}(\mathbb{T})=	\max\left\{\max_{\substack{1\leq i\leq \frac{n-3}{4}}}\left\{w_A(T_{2i}+S_{2i}),w_A(T_{2i}-S_{2i})\right\},w_A(T_{\frac{n+1}{2}}), \max_{\substack{\frac{n+5}{4}\leq i\leq \frac{n+1}{2}}}\left\{w_A(T_{2i-1}+S_{2i-1}),w_A(T_{2i-1}-S_{2i-1})\right\} \right\}.
      \end{align*}}
      \end{enumerate}
  \end{cor}

\begin{remark}
   For $S_i\in \mathcal{B}_A(\mathcal{H})$ with $S_{n-(2i-1)}=S_{2i} ~and~ S_{n-(2i-2)}=S_{2i-1},  i=1,\dots, n$, we have 
    $$w_A(off-diag(S_1,\dots, S_n))=\displaystyle\max_{1\leq i\leq n}w_A(S_{2i}).$$
\end{remark}

  The following result is an interesting formula for a cross-diagonal block matrix with different unitary matrix. But note that using the same unitary matrix it is not easy to calculate the $\mathbb{A}$-numerical radius of the matrix mentioned in Theorem \ref{Thm1}.
 \begin{remark}\label{Rem2.2}

Let $ \mathbb{R}=\begin{bmatrix}
		T & O &  & \cdots&  & O&  S\\
		O  & T& & & & S & O\\
		\vdots  &  &\ddots &&\adots  & & \vdots\\
  	  &  &\adots& &\ddots  & &  \\
		O &S &  & && T & O \\
		S  & O  & &\cdots&  &O & T
		\end{bmatrix}$ \[=
    \left[\begin{array}{@{}c|c@{}}
   diag(T,\dots, T) &  off-diag(S,\dots, S)  \\ \hline
    off-diag(S,\dots, S)  &   diag(T,\dots, T)
    \end{array}\right].
    \] If $n$ is even, then,
  $$w_{\mathbb{A}}(\mathbb{R})=\max\{w_A(T+S), w_A(T-S) \}.$$
 and if n is odd, then
\begin{enumerate}
    \item [(i)]  $w_{\mathbb{A}}(\mathbb{R})=\max\{w_A(T+S), w_A(T), w_A(T-S) \}.$
    \item [(ii)]  $w_{\mathbb{A}}(\mathbb{R})=\max\{w_A(T+S), w_A(S), w_A(T-S) \}.$ 
\end{enumerate}
 \end{remark}
 \begin{proof}
    For $n$ even, let
    \[\mathbb{U}=
   \frac{1}{\sqrt{2}} \left[\begin{array}{@{}c|c@{}}
   diag(I,\dots, I) &  off-diag(I,\dots, I)  \\ \hline
    off-diag(I,\dots, I)  &   -diag(I,\dots, I)
    \end{array}\right],
    \]
    $$\mathbb{U}\mathbb{R}^{\#_{\mathbb{A}}}\mathbb{U}^{\#_{\mathbb{A}}}=diag(T^{\#_A}+S^{\#_A}, \dots, T^{\#_A}+S^{\#_A}, T^{\#_A}-S^{\#_A}, \dots, T^{\#_A}-S^{\#_A}).$$
Now, using the Equation \ref{eqn_00000001.4}, and $ w_{\mathbb{A}}(\mathbb{R})=w_{\mathbb{A}}(\mathbb{R}^{\#_{\mathbb{A}}})=w_{\mathbb{A}}( \mathbb{U}^{\#_{\mathbb{A}}} \mathbb{R}^{\#_{\mathbb{A}}}  \mathbb{U} )$, we have
    \begin{align*}
        w_{\mathbb{A}}(\mathbb{R})&=w_A\{diag(T+S, \dots, T+S, T-S, \dots, T-S)\}\\
        &=\max\{w_A(T+S), w_A(T-S) \}.
    \end{align*}
    The proof is similar for $n$ odd with the unitary matrix
     \[\mathbb{U}=
    \frac{1}{\sqrt{2}}\left[\begin{array}{@{}ccc@{}}
   diag(I,\dots, I) & O&  off-diag(I,\dots, I)  \\ 
   O &\sqrt{2}I&O\\
    off-diag(I,\dots, I)  & O&  -diag(I,\dots, I)
    \end{array}\right].
    \]
 \end{proof}
\begin{remark}
    If $T=S$, we have  $w_{\mathbb{A}}(\mathbb{R})=2w_A(T)$. 
\end{remark}
\begin{remark}
    For $S\in \mathcal{B}_A(\mathcal{H})$, and taking $T=0$ in Remark \ref{Rem2.2}, we get\\
    $w_A(off-diag(S,\dots, S))=w_A(S)$, which is a generalization of the Lemma  \ref{lem0001} (iv).  
\end{remark}

    \begin{remark}\label{Remark2.4}
From \cite[Theorem 2.1]{KITSAT}, for $T, S\in \mathcal{B}_A(\mathcal{H})$, we have  \\
       $  w_{\mathbb{A}}\left(\begin{bmatrix}
		T & S &  & \cdots&  & S&  S\\
		S  & T& & & & S & S\\
		\vdots  &  &\ddots &&\adots  & & \vdots\\
  	  &  &\adots& &\ddots  & &  \\
		S &S &  & && T & S \\
		S  & S  & &\cdots&  &S & T
		\end{bmatrix}_{n\times n}\right)=\max\{w_A(T+(n-1)S), w_A(T-S)\}.$ 
      \end{remark}
In the following theorem, we generalize these facts when $n$ is even. 
\begin{theorem}\label{Thm2.5}
    Let $R, S, T \in \mathcal{B}_A(\mathcal{H})$, $\mathbb{M}=\begin{bmatrix}
		R & S &  & \cdots&  & S&  T\\
		S  & R& & & & T & S\\
		\vdots  &  &\ddots &&\adots  & & \vdots\\
  	  &  &\adots& &\ddots  & &  \\
		S &T &  & && R & S \\
		T  & S  & &\cdots&  &S & R
		\end{bmatrix}_{n\times n}$ and  $n$ is \\an even integer.\\ Then
  $  w_{\mathbb{A}}\left(\mathbb{M}\right)=\max\{w_A(R+T+(n-2)S),w_A(R+T-2S), w_A(R-T)\}.$ 
\end{theorem}
\begin{proof}
Using the same $\mathbb{A}$-unitary operator $\mathbb{U}$ as used in Theorem \ref{Thm1}, we have, 
$$
\mathbb{U}^{\#_{\mathbb{A}}} \mathbb{M}^{\#_{\mathbb{A}}}  \mathbb{U}=\begin{bmatrix}
		\begin{bmatrix}
   R &T\\
   T &R
\end{bmatrix} & \begin{bmatrix}
   S & S\\
   S & S
\end{bmatrix} & \cdots&   \begin{bmatrix}
   S & S\\
   S & S
\end{bmatrix}\\
		 \begin{bmatrix}
   S & S\\
   S & S
\end{bmatrix} & \ddots &\ddots&  \vdots\\
		\vdots  &\ddots& \ddots  &  \vdots\\
		\begin{bmatrix}
   S & S\\
   S & S
\end{bmatrix} & \cdots &\cdots  &\begin{bmatrix}
   R &T\\
   T &R
\end{bmatrix} 
		\end{bmatrix}_{\frac{n}{2}\times\frac {n}{2}}^{\#_{\mathbb{A}}}
.$$ 
Now, using the fact that $w_{\mathbb{A}}(\mathbb{T})=w_{\mathbb{A}}(\mathbb{U}^{\#_{\mathbb{A}}}\mathbb{TU})$ for any $\mathbb{T}\in\mathcal{B}_A(\mathcal{H}),$  and Remark \ref{Remark2.4}, we have
\begin{align*}
    w_{\mathbb{A}}(\mathbb{M})=w_{\mathbb{A}}(\mathbb{M}^{\#_{\mathbb{A}}})&=w_{\mathbb{A}}( \mathbb{U}^{\#_{\mathbb{A}}} \mathbb{M}^{\#_{\mathbb{A}}}  \mathbb{U} )=w_{\mathbb{A}}\left(\begin{bmatrix}
		\begin{bmatrix}
   R &T\\
   T &R
\end{bmatrix} & \begin{bmatrix}
   S & S\\
   S & S
\end{bmatrix} & \cdots&   \begin{bmatrix}
   S & S\\
   S & S
\end{bmatrix}\\
		 \begin{bmatrix}
   S & S\\
   S & S
\end{bmatrix} & \ddots &\ddots&  \vdots\\
		\vdots  &\ddots& \ddots  &  \vdots\\
		\begin{bmatrix}
   S & S\\
   S & S
\end{bmatrix} & \cdots &\cdots  &\begin{bmatrix}
   R &T\\
   T &R
\end{bmatrix} 
		\end{bmatrix}_{\frac{n}{2}\times\frac {n}{2}}\right)\\
  &=\max\left\{w_A\left(\begin{bmatrix}
   R &T\\
   T &R
\end{bmatrix}+\left(\frac{n}{2}-1\right)\begin{bmatrix}
   S & S\\
   S & S
\end{bmatrix}\right), w_A\left(\begin{bmatrix}
   R &T\\
   T &R
\end{bmatrix}-\begin{bmatrix}
   S & S\\
   S & S
\end{bmatrix}\right)\right\}\\
&=\max\{w_A(R+T+(n-2)S),w_A(R+T-2S), w_A(R-T)\},
\end{align*}
where the last equality follows from Lemma \ref{lem0001} (iv).
\end{proof}

It is a natural question in our mind what will happen for $n$ odd?

\begin{remark}
    For $n$ odd, it is not possible to partition the matrix mentioned in Theorem \ref{Thm2.5} into $2\times 2$ blocks, so we can not proceed for $n$ odd as in the case of $n$ even.  For $\mathbb{M}=\begin{bmatrix}
   2I & 3I &2I\\
   3I & 2I &3I\\
   2I & 3I &2I
\end{bmatrix}$, $\mathbb{A}=\begin{bmatrix}
   I & 0 &0\\
   0 & I &0\\
   0 & 0 &I
\end{bmatrix}$, we have $w_{\mathbb{A}}(\mathbb{M})\approx 7.358899$, whereas the right-hand side of Theorem \ref{Thm2.5} is equal to $7$. On the other hand, For $\mathbb{M}=\begin{bmatrix}
   3I & 2I &3I\\
   2I & 3I &2I\\
   3I & 2I &3I
\end{bmatrix}$, with same $\mathbb{A}$, we have $w_{\mathbb{A}}(\mathbb{M})\approx 7.7016$, whereas the right-hand side of Theorem \ref{Thm2.5} is equal to $8$. So, for the $n$ odd case, the two sides of Theorem \ref{Thm2.5} are not comparable in general.
\end{remark}

\section{$ \mathbb{A}$-numerical radii of $n\times n$ left circulant and skew left circulant operator matrices}
Recall that $\mbox{lcirc}(T_1,\dots,T_n)=\mbox{circ}(T_n,\dots,T_1)J,$ 
 where $J=$ off-diag$(I, \dots, I)$ and\\ $ \mbox{slcirc}(T_1,\dots,T_n)=\mbox{scirc}(T_n,\dots,T_1)J$ with the same $J$. So, all the results in \cite{KITSAT} concerning the $\mathbb{A}$-operator norms of circulant and skew circulant operator matrices are carried out to the left circulant and skew left circulant operator matrices. However, the results in \cite{KITSAT} related to the $\mathbb{A}$-numerical radius cannot be transferred to left circulant and skew left circulant operator matrices; for example, see Remark \ref{Remark3.4}.

The aim of this section is to present certain $\mathbb{A}$-numerical radius equalities and inequalities for $n\times n$ left circulant and skew left circulant operator matrices.
  
\begin{theorem}\label{Thm2}
		Let $T_i\in \mathcal{B}_A(\mathcal{H})$ for $1\leq i\leq n$. Then 
		\begin{align*}
		w_\mathbb{A}(\mathbb{T}_{lcirc})=w_{\mathbb{A}}\left(\begin{bmatrix}
		 \sum_{i=1}^{n}T_{i} & O  &O &   \cdots & O\\
		 O  & O  &O&  \cdots &  \sum_{i=1}^{n}{w}^{(n-1)(1-i)}T_{i}  \\
   O &  O &  O& \adots & O\\
		 \vdots & \vdots &\sum_{i=1}^{n}{w}^{2(1-i)}T_{i}& \vdots & \vdots \\
		 O  &\sum_{i=1}^{n}{w}^{(1-i)}T_{i}   &O & \cdots &   O
		 \end{bmatrix}\right).
		\end{align*}
	\end{theorem}
	\begin{proof}
		Let $\mathbb{T}_{lcirc}=\begin{bmatrix}
		T_{1} & T_{2} & \cdots & T_{n-1}   &   T_{n}\\
		T_{2}  &  \cdots &  T_{n-1}& T_{n} & T_{1}\\
		\vdots & \adots &\adots & \adots & \vdots \\
  T_{n-1}  & T_{n} &  T_{1}&\cdots  &  T_{n-2}\\
		T_{n}  & T_{1}  &\cdots&  T_{n-2} & T_{n-1}
		\end{bmatrix},$ 
		let $1,\omega,\omega^2,\dots, \omega^{n-1}$ be $n$ roots of unity with $\omega=e^{\frac{2\pi i}{n}}$
		
		and $\mathbb{U}=\frac{1}{\sqrt{n}}\begin{bmatrix}
		I & I& I  &  \cdots  &   I\\
		I & \omega I& \omega^2I & \cdots  &   \omega^{n-1}I\\
		I & \omega^2I& \omega^4I &  \cdots  &   \omega^{n-2}I\\
		\vdots & \vdots &\vdots &\ddots& \vdots\\
		I & \omega^{n-1}I& \omega^{n-2}I  &  \cdots  &   \omega I\\
		\end{bmatrix}.$
		
		It can be observed that  that $ \bar{\omega}=\omega^{n-1}, \bar{\omega}^2=\omega^{n-2},\cdots, \bar{\omega}^k=\omega^{n-k}$, $k=0, 1,\dots, n-1$, so
		
		$\mathbb{U}^{\#_{\mathbb{A}}}=\frac{1}{\sqrt{n}}\begin{bmatrix}
		P_{\overline{\mathcal{R}(A)}} & P_{\overline{\mathcal{R}(A)}}& P_{\overline{\mathcal{R}(A)}}  & \cdots  &   P_{\overline{\mathcal{R}(A)}}\\
		P_{\overline{\mathcal{R}(A)}} & \omega^{n-1}P_{\overline{\mathcal{R}(A)}}& \omega^{n-2}P_{\overline{\mathcal{R}(A)}}  & \cdots & \omega P_{\overline{\mathcal{R}(A)}}\\
		P_{\overline{\mathcal{R}(A)}} & \omega^{n-2}P_{\overline{\mathcal{R}(A)}}& \omega^{n-4}P_{\overline{\mathcal{R}(A)}}  & \cdots  &  \omega^2P_{\overline{\mathcal{R}(A)}}\\
		\vdots & \vdots &\vdots & \ddots &\vdots\\
		P_{\overline{\mathcal{R}(A)}}&\omega P_{\overline{\mathcal{R}(A)}}& \omega^2P_{\overline{\mathcal{R}(A)}}  & \cdots  &  \omega^{n-1}P_{\overline{\mathcal{R}(A)}}\\
		\end{bmatrix}$ and
		
		$\mathbb{UU}^{\#_{\mathbb{A}}}= \begin{bmatrix}
		P_{\overline{\mathcal{R}(A)}} & O  & \cdots  &   O\\
		O  & P_{\overline{\mathcal{R}(A)}} & \cdots  & O\\
		\vdots & \vdots &\ddots & \vdots \\
		O  & O  &\cdots&  P_{\overline{\mathcal{R}(A)}}
		\end{bmatrix}=\mathbb{U}^{\#_{\mathbb{A}}}\mathbb{U}.$
		
		 Now, using Lemma \ref{lemma1}, and the fact that $w_{\mathbb{A}}(\mathbb{T})=w_{\mathbb{A}}(\mathbb{UTU}^{\#_{\mathbb{A}}})$ for any $\mathbb{T}\in\mathcal{B}_A(\mathcal{H}),$
		we get		
		\begin{align*}
		w_{\mathbb{A}}(\mathbb{T}_{lcirc})&=w_{\mathbb{A}}(\mathbb{T}_{lcirc}^{\#_{\mathbb{A}}})=w_{\mathbb{A}}(\mathbb{UT}_{lcirc}^{\#_{\mathbb{A}}}\mathbb{U}^{\#_{\mathbb{A}}})\\
  &= w_{\mathbb{A}}\left( \begin{bmatrix}
		\sum_{i=1}^{n}T_{i}^{\#_A} & O  & \cdots &   O & O\\
		O  & O  &\cdots&  O &  \sum_{i=1}^{n}w^{(i-1)}T_{i}^{\#_A}\\
  O &  O &  \cdots & \sum_{i=1}^{n}w^{2(i-1)}T_{i}^{\#_A} & O\\
		\vdots & \vdots &\adots & \vdots & \vdots \\
		O  & \sum_{i=1}^{n}w^{(n-1)(i-1)}T_{i}^{\#_A}  &\cdots & O &   O
		\end{bmatrix}\right)\\
  &=w_{\mathbb{A}}\left(\begin{bmatrix}
		 \sum_{i=1}^{n}T_{i} & O  & \cdots &   O & O\\
		 O  & O  &\cdots&  O &  \sum_{i=1}^{n}{w}^{(i-1)}T_{i}  \\
   O &  O &  \cdots & \sum_{i=1}^{n}{w}^{2(i-1)}T_{i} & O\\
		 \vdots & \vdots &\adots & \vdots & \vdots \\
		 O  &\sum_{i=1}^{n}{w}^{(n-1)(i-1)}T_{i}   &\cdots & O &   O
		 \end{bmatrix}^{\#_{\mathbb{A}}}\right)\\
&=w_{\mathbb{A}}\left(\begin{bmatrix}
		 \sum_{i=1}^{n}T_{i} & O  & O &  \cdots & O\\
		 O  & O  &O&  \cdots &  \sum_{i=1}^{n}{w}^{(n-1)(1-i)}T_{i}  \\
   O &  O &  O & \adots & O\\
		 \vdots & \vdots &\sum_{i=1}^{n}{w}^{2(1-i)}T_{i}& \vdots & \vdots \\
		 O  &\sum_{i=1}^{n}{w}^{(1-i)}T_{i}   &O & \cdots &   O
		 \end{bmatrix}^{\#_{\mathbb{A}}}\right)\\
  &=w_{\mathbb{A}}\left(\begin{bmatrix}
		 \sum_{i=1}^{n}T_{i} & O  &O &   \cdots & O\\
		 O  & O  &O&  \cdots &  \sum_{i=1}^{n}{w}^{(n-1)(1-i)}T_{i}  \\
   O &  O &  O& \adots & O\\
		 \vdots & \vdots &\sum_{i=1}^{n}{w}^{2(1-i)}T_{i}& \vdots & \vdots \\
		 O  &\sum_{i=1}^{n}{w}^{(1-i)}T_{i}   &O & \cdots &   O
		 \end{bmatrix}\right).
				\end{align*}
		
	\end{proof}

\begin{rem}
    If $T_i\in \mathcal{B}_A(\mathcal{H})$ for $1\leq i\leq n$, then $w_A(lcirc(T_1, \dots, T_n))=w_A(lcirc(T_n, \dots, T_1))$.
\end{rem}

\begin{cor}

	Let $T_i\in \mathcal{B}_A(\mathcal{H})$ for $1\leq i\leq n$. Then 

 \begin{enumerate}
     \item [(i)] If $n$ is even, we have
          \begin{align}\label{}
	w_\mathbb{A}(\mathbb{T}_{lcirc})\leq 	w_A\left( \sum_{i=1}^{n}T_{i}    \right) + \max\left\{w_A\left(D_{j}    \right)~j=1,2,\dots, n-1 \right\}
      \end{align}
        \item [(ii)] If $n$  is odd, we have
          \begin{align}\label{}
	w_\mathbb{A}(\mathbb{T}_{lcirc})\leq \max\left\{w_A\left( \sum_{i=1}^{n}T_{i}    \right), w_A(D_{\frac{n}{2}}) \right\} + \max\left\{w_A\left(D_{j}    \right)~j=1,2,\dots, n-1 ~\mbox{and}~ j\neq \frac{n}{2} \right\},
      \end{align}
  \end{enumerate}
where $D_j=\sum_{i=1}^{n}{w}^{j(1-i)}T_{i}, 1\leq j\leq n-1$.
\end{cor}

\begin{proof}

\underline{Case-1, $n$ is even}:
Using Theorem \ref{Thm2}, we have
    \begin{align*}
		w_\mathbb{A}(\mathbb{T}_{lcirc})&=w_{\mathbb{A}}\left(\begin{bmatrix}
		\sum_{i=1}^{n}T_{i} & O  & \cdots &   O & O\\
		O  & O  &\cdots&  O &  D_{n-1} \\
  O &  O &  \cdots & .... & O\\
		\vdots & \vdots &\adots & \vdots & \vdots \\
		O  &D_1   &\cdots & O &   O
		\end{bmatrix}\right)\\
  &\leq w_{\mathbb{A}}\left(\begin{bmatrix}
		\sum_{i=1}^{n}T_{i} & O  & \cdots &   O & O\\
		O  & O  &\cdots&  O &  O  \\
  O &  O &  \cdots & O & O\\
		\vdots & \vdots &\adots & \vdots & \vdots \\
		O  &O   &\cdots & O &   O
		\end{bmatrix}\right)+w_{\mathbb{A}}\left(\begin{bmatrix}
		O & O  & \cdots &   O & O\\
		O  & O  &\cdots&  O &  D_{n-1} \\
  O &  O &  \cdots & .... & O\\
		\vdots & \vdots &\adots & \vdots & \vdots \\
		O  & D_{1}   &\cdots & O &   O
		\end{bmatrix}\right)\\
  &\leq w_A\left( \sum_{i=1}^{n}T_{i}    \right) + w_{\mathbb{A}}\left(\begin{bmatrix}
		O & O  & \cdots &   O & O\\
		O  & O  &\cdots&  O &   w_A\left( D_{n-1}\right)  \\
  O &  O &  \cdots & .... & O\\
		\vdots & \vdots &\adots & \vdots & \vdots \\
		O  & w_A\left( D_1\right)   &\cdots & O &   O
		\end{bmatrix}\right)\\
  &(\mbox{by Lemma} \ref{LEMA1.7} ~\mbox{with the identity} D_j=D_{n-j})\\
  &=w_A\left( \sum_{i=1}^{n}T_{i}    \right) + \rho_{\mathbb{A}}\left(\begin{bmatrix}
		O & O  & \cdots &   O & O\\
		O  & O  &\cdots&  O &   w_A\left( D_{n-1}\right)  \\
  O &  O &  \cdots & .... & O\\
		\vdots & \vdots &\adots & \vdots & \vdots \\
		O  & w_A\left( D_1\right)   &\cdots & O &   O
		\end{bmatrix}\right)\\
  &=w_A\left( \sum_{i=1}^{n}T_{i}    \right) + \left(\rho_{\mathbb{A}}\left(\begin{bmatrix}
		O & O  & \cdots &   O & O\\
		O  & w_A^2\left( D_{n-1}\right)  &\cdots&  O &   O  \\
  O &  O &  w_A^2\left( D_{n-2}\right) & \dots & O\\
		\vdots & \vdots &\adots & \ddots & \vdots \\
		O  & O   &\cdots & O &  w_A^2\left( D_1\right)
		\end{bmatrix}\right)\right)^\frac{1}{2}\\
  &\hspace{7cm}(\mbox{by Lemma} \ref{l87}~ \mbox{k=2} )\\
   &=w_A\left( \sum_{i=1}^{n}T_{i}    \right) + \max\left\{w_A\left(D_{j}    \right)~j=1,2,\dots, n-1 \right\}.
		\end{align*}
  \underline{Case-2, $n$ is odd}:
Using Theorem \ref{Thm2}, we have
    \begin{align*}
		w_\mathbb{A}(\mathbb{T}_{lcirc})&=w_{\mathbb{A}}\left(\begin{bmatrix}
		\sum_{i=1}^{n}T_{i} & O  & \cdots &   O & O& O\\
		O  & O  &\cdots&  O & O& D_{n-1} \\
  O &  O &  \ddots & \adots  & \adots& O\\
   O &  O & & D_{\frac{n}{2}}  & \adots  & O\\
		\vdots & \vdots &\adots & \vdots & \ddots&\vdots \\
		O  &D_1   &\cdots & O & O&  O
		\end{bmatrix}\right)\\
  &\leq w_{\mathbb{A}}\left(\begin{bmatrix}
		\sum_{i=1}^{n}T_{i} & O  & \cdots &   O & O& O\\
		O  & O  &\cdots&  O & O& O \\
  O &  O &  \ddots & \adots  & \adots& O\\
   O &  O & & D_{\frac{n}{2}}  & \adots  & O\\
		\vdots & \vdots &\adots & \vdots & \ddots&\vdots \\
		O  & O  &\cdots & O & O&  O
		\end{bmatrix}\right)+w_{\mathbb{A}}\left(\begin{bmatrix}
		O & O  & \cdots &   O & O& O\\
		O  & O  &\cdots&  O & O& D_{n-1} \\
  O &  O &  \ddots & \adots  & \adots& O\\
   O &  O & & O  & \adots  & O\\
		\vdots & \vdots &\adots & \vdots & \ddots&\vdots \\
		O  &D_1   &\cdots & O & O&  O
		\end{bmatrix}\right)\\
  &\leq \max\left\{w_A\left( \sum_{i=1}^{n}T_{i}    \right), w_A(D_{\frac{n}{2}}) \right\}+ w_{\mathbb{A}}\left(\begin{bmatrix}
	O & O  & \cdots &   O & O& O\\
		O  & O  &\cdots&  O & O&w_A\left( D_{n-1}\right) \\
  O &  O &  \ddots & \adots  & \adots& O\\
   O &  O & & O  & \adots  & O\\
		\vdots & \vdots &\adots & \vdots & \ddots&\vdots \\
		O  &w_A\left( D_{1}\right)   &\cdots & O & O&  O
		\end{bmatrix}\right)\\
  &(\mbox{by Lemma} \ref{LEMA1.7} ~\mbox{with the identity} D_j=D_{n-j} )\\
  &=\max\left\{w_A\left( \sum_{i=1}^{n}T_{i}    \right), w_A(D_{\frac{n}{2}}) \right\}+ \rho_{\mathbb{A}}\left(\begin{bmatrix}
	O & O  & \cdots &   O & O& O\\
		O  & O  &\cdots&  O & O&w_A\left( D_{n-1}\right) \\
  O &  O &  \ddots & \adots  & \adots& O\\
   O &  O & & O  & \adots  & O\\
		\vdots & \vdots &\adots & \vdots & \ddots&\vdots \\
		O  &w_A\left( D_{1}\right)   &\cdots & O & O&  O
		\end{bmatrix}\right)\\
  &=\max\left\{w_A\left( \sum_{i=1}^{n}T_{i}    \right), w_A(D_{\frac{n}{2}}) \right\} + \left(\rho_{\mathbb{A}}\left(\begin{bmatrix}
		O & O  & \cdots &   O & O& O\\
		O  &  w_A^2\left( D_{n-1}\right)  &\cdots&  O & O&O \\
  O &  O & \ddots & \adots  & \adots& O\\
   O &  O & & O  & \adots  & O\\
		\vdots & \vdots &\adots & \vdots & \ddots&\vdots \\
		O  &O  &\cdots & O & O& w_A^2\left( D_{1}\right)
		\end{bmatrix}\right)\right)^\frac{1}{2}\\
  &~~~~~~~~~~~~~~~~~~~~~~~~~~~~~~~~~~~~~~~~~~~~~~~~~~~~~~~~~~~~~~~~~~~~~~~~~~~~~(\mbox{by Lemma} \ref{l87}~ \mbox{k=2} )
  \end{align*}
  \begin{align*}
   &=\max\left\{w_A\left( \sum_{i=1}^{n}T_{i}    \right), w_A(D_{\frac{n}{2}}) \right\} + \max\left\{w_A\left(D_{j}    \right)~j=1,2,\dots, n-1 ~\mbox{and}~ j\neq \frac{n}{2} \right\}.
		\end{align*}
\end{proof}

Motivated by the work of  \cite{KITSAT}, 
in the following theorem, we give a counterpart representation of the skew left circulant operator matrices.
\begin{theorem}\label{Thm3}
		Let $T_i\in \mathcal{B}_A(\mathcal{H})$ for $1\leq i\leq n$. Then
		$$ w_{\mathbb{A}}(\mathbb{T}_{slcirc})=w_{\mathbb{A}}\left(\begin{bmatrix}
			O & O  & \cdots &   O &\sum_{i=1}^n(\sigma\omega^0)^{i+1} T_i\\
		O  & O  &\cdots&  \sum_{i=1}^n(\sigma\omega)^{i+1} T_i &  O  \\
  O &  O &  \cdots & .... & O\\
		\vdots & \vdots &\adots & \vdots & \vdots \\
		\sum_{i=1}^n(\sigma\omega^{n-1})^{i+1} T_i &O  &\cdots & O &   O
		\end{bmatrix}\right)$$  
	\end{theorem}

	\begin{proof}
		The $n$ roots of the equation $z^n=-1$ are $\sigma, \sigma \omega, \sigma \omega^2, \dots, \sigma \omega^{n-1}$, where $\sigma=e^{\frac{\pi i}{n}}$ and $\omega=e^{\frac{2\pi i}{n}}$.
		Let $\mathbb{T}_{slcirc}=\begin{bmatrix}
		T_{1} & T_{2} & \cdots & T_{n-1}   &   T_{n}\\
		T_{2}  &  \cdots &  T_{n-1}& T_{n} & -T_{1}\\
		\vdots & \adots &\adots & \adots & \vdots \\
  T_{n-1}  & T_{n} &  -T_{1}&\cdots  &  -T_{n-2}\\
		T_{n}  & -T_{1}  &\cdots&  -T_{n-2} & -T_{n-1}
		\end{bmatrix}$ 
		and
		\begin{align*}
		\mathbb{U}=\frac{1}{\sqrt{n}}\begin{bmatrix}
		I & \sigma I&\sigma^2 I  & \cdots  &  \sigma^{n-1} I\\
		(\sigma \omega)I & (\sigma \omega)^2 I& (\sigma \omega)^3I  & \cdots  &   (\sigma \omega)^nI\\
		\vdots & \vdots &\vdots & \vdots &\vdots\\
		(\sigma \omega^{n-2})^{n-2}I &  (\sigma \omega^{n-2})^{n-1}I&  (\sigma \omega^{n-2})^nI  & \cdots  &  (\sigma \omega^{n-2})^{2n-1}I\\
		(\sigma \omega^{n-1})^{n-1}I &  (\sigma \omega^{n-1})^{n}I&  (\sigma \omega^{n-1})^{n+1}I  & \cdots  &   (\sigma \omega^{n-1})^{2n-2}I
		\end{bmatrix}.
		\end{align*}
 
		Using a similar argument as used in the Theorem \ref{Thm1}, we can show that $\mathbb{U}$ is $\mathbb{A}$-unitary.
		Now, using Lemma \ref{lemma1},  and the property $w_{\mathbb{A}}(\mathbb{T})=w_{\mathbb{A}}(\mathbb{UTU}^{\#_{\mathbb{A}}})$ for any $\mathbb{T}\in\mathcal{B}_A(\mathcal{H})$,
		we get
		\begin{align*}
		w_{\mathbb{A}}(\mathbb{T}_{slcirc})&=w_{\mathbb{A}}(\mathbb{T}_{slcirc}^{\#_{\mathbb{A}}})=w_{\mathbb{A}}(\mathbb{UT}_{slcirc}^{\#_{\mathbb{A}}}\mathbb{U}^{\#_{\mathbb{A}}})\\
 & =w_{\mathbb{A}}\left(\begin{bmatrix}
		O & O  & \cdots &   O &\sum_{i=1}^n(\overline{\sigma\omega^{n-1}})^{i+1} T_i^{\#_A}\\
		O  & O  &\cdots&  \sum_{i=1}^n(\overline{\sigma\omega^{n-2}})^{i+1} T_i^{\#_A} &  O  \\
  O &  O &  \cdots & .... & O\\
		\vdots & \vdots &\adots & \vdots & \vdots \\
		\sum_{i=1}^n(\overline{\sigma\omega^0})^{i+1} T_i^{\#_A} &O  &\cdots & O &   O
		\end{bmatrix}\right)\\
  \end{align*}
  \begin{align*}
  &=w_{\mathbb{A}}\left(\begin{bmatrix}
		O & O  & \cdots &   O &\sum_{i=1}^n(\sigma\omega^0)^{i+1} T_i\\
		O  & O  &\cdots&  \sum_{i=1}^n(\sigma\omega)^{i+1} T_i &  O  \\
  O &  O &  \cdots & .... & O\\
		\vdots & \vdots &\adots & \vdots & \vdots \\
		\sum_{i=1}^n(\sigma\omega^{n-1})^{i+1} T_i &O  &\cdots & O &   O
		\end{bmatrix}^{\#_{\mathbb{A}}}\right)\\
  &=w_{\mathbb{A}}\left(\begin{bmatrix}
			O & O  & \cdots &   O &\sum_{i=1}^n(\sigma\omega^0)^{i+1} T_i\\
		O  & O  &\cdots&  \sum_{i=1}^n(\sigma\omega)^{i+1} T_i &  O  \\
  O &  O &  \cdots & .... & O\\
		\vdots & \vdots &\adots & \vdots & \vdots \\
		\sum_{i=1}^n(\sigma\omega^{n-1})^{i+1} T_i &O  &\cdots & O &   O
		\end{bmatrix}\right).
		\end{align*}
	\end{proof}
	\begin{cor}
	  Let $ T_i \in \mathcal{B}_A(\mathcal{H}),~ i= 1, 2,\ldots, n $. 
     If $n$ is even, then
     \begin{align*}
            w_{\mathbb{A}}(\mathbb{T}_{slcirc})
            \leq \frac{n}{2} \displaystyle\sum_{i=1}^{n} \|T_i \|_A,
        \end{align*}
        and if $n$ is odd, then
 \begin{align*}
   w_{\mathbb{A}}(\mathbb{T}_{slcirc})
   \leq \sum_{i=1}^n w_A( T_i)+\frac{n-1}{2} \displaystyle\sum_{i=1}^{n} \|T_i\|_A .  
 \end{align*}
	\end{cor}
 \begin{proof}
        Using Lemma \ref{thm3.4}, and  Theorem \ref{Thm3}, we see that
        if $n$ is even, then
        \begin{align*}
            w_{\mathbb{A}}(\mathbb{T}_{slcirc})&\leq \frac{1}{2} \displaystyle\sum_{j=0}^{n-1} \|\sum_{i=1}^n(\sigma\omega^{j})^{i+1} T_i \|_A\\
            &\leq \frac{n}{2} \displaystyle\sum_{i=1}^{n} \|T_i \|_A,
        \end{align*}

 and if $n$ is odd, then
 \begin{align*}
   w_{\mathbb{A}}(\mathbb{T}_{slcirc})&\leq  w_A\bigg(\sum_{i=1}^n(\sigma\omega^\frac{n-1}{2})^{i+1} T_i\bigg)+\frac{n-1}{2} \displaystyle\sum_{i=1}^{n} \|T_i\|_A\\
   &\leq \sum_{i=1}^n w_A( T_i)+\frac{n-1}{2} \displaystyle\sum_{i=1}^{n} \|T_i\|_A .  
 \end{align*}

 \end{proof}

\begin{rem}
    If $T_i\in \mathcal{B}_A(\mathcal{H})$ for $1\leq i\leq n$, then $w_A(slcirc(T_1, \dots, T_n))=w_A(slcirc(T_n, \dots, T_1))$.
\end{rem}

    \begin{remark}\label{Remark3.4}
From \cite[Theorem 2.2]{KITSAT}, for $T, S\in \mathcal{B}_A(\mathcal{H})$, $ \sigma=e^{\frac{\pi i}{n}}$ and for skew circulant operator matrices, we have  \\
       $  w_{\mathbb{A}}\left(\begin{bmatrix}
		T & \sigma T &  & \cdots&  & ..&  \sigma^{n-1}T\\
		-\sigma^{n-1}T  & T& & & &T & T\\
		\vdots  &  &\ddots &&\adots  & & \vdots\\
  	  &  &\adots& &\ddots  & &  \\
		. &.&  & && & \sigma T \\
		-\sigma T  & .  & &\cdots&  & -\sigma^{n-1} T & T
		\end{bmatrix}_{n\times n}\right)=nw_A(T).$ 
      \end{remark}
      which yields 
      \begin{align*}
          w_A(slcirc(T, \sigma T, \dots, \sigma^{n-1}T))&=w_A(slcirc(\sigma^{n-1} T, \dots, \sigma T,  T))\\
         &=w_A(scirc( T, \sigma T,\dots,   \sigma^{n-1}T)J)\\
         &\leq w_A(scirc( T, \sigma T,\dots,   \sigma^{n-1}T)).
      \end{align*}
      In general, they are not equal. However, equality can hold in some cases; for example, when the operator is a nilpotent operator of index $2$.
      
       \section{$ \mathbb{A}$-numerical radii of certain imaginary left circulant and imaginary skew left circulant operator matrices}
 \begin{prop}\label{p8}
		Let $T_i\in \mathcal{B}_A(\mathcal{H})$ for $1\leq i\leq 2$. Then
		$$ w_{\mathbb{A}}\left(\begin{bmatrix}
			T_1 & T_2\\
			T_2 & iT_1
			\end{bmatrix}\right)\leq w_A(T_1)+w_A(T_2).$$
	\end{prop}

\begin{cor}
    Let $T_i\in \mathcal{B}_A(\mathcal{H})$ for $1\leq i\leq 2$. Then
		$$ w_{\mathbb{A}}\left(\begin{bmatrix}
			T_1 & T_2\\
			T_2 & -iT_1
			\end{bmatrix}\right)\leq w_A(T_1)+w_A(T_2).$$
\end{cor}

 \begin{prop}\label{p9}
		Let $T_i\in \mathcal{B}_A(\mathcal{H})$ for $1\leq i\leq 3$. Then
		\begin{align*}
	w_{\mathbb{A}}\left(\begin{bmatrix}
			T_1 & T_2& T_3\\
			T_2 & T_3&iT_1\\
   T_3&iT_1&iT_2
			\end{bmatrix}\right)&\leq   w_A(T_3)+\max\{ w_A(T_1),  w_A(T_2)\}+ \sqrt{2}\left(\| T_1-T_2\|_A+\| T_1+T_2\|_A\right).
 \end{align*} 
  
	\end{prop}
 \begin{proof}
   Let $\mathbb{V}=\frac{1}{\sqrt{2}}\begin{bmatrix}
	iI&O & I\\
 O& \sqrt{2}I&O\\
	I & O&iI
	 \end{bmatrix}$.  
 Now, using the fact that $w_A(U^{\#_A}TU)=w_A(T)$ and the triangle inequality for the $\mathbb{A}$-numerical radius, we have
\begin{align*}
   & w_{\mathbb{A}}\left(\begin{bmatrix}
			T_1 & T_2& T_3\\
			T_2 & T_3&iT_1\\
   T_3&iT_1&iT_2
			\end{bmatrix}\right)=   w_{\mathbb{A}}\left(\mathbb{V}^{\#_{\mathbb{A}}}\begin{bmatrix}
			T_1 & T_2& T_3\\
			T_2 & T_3&iT_1\\
   T_3&iT_1&iT_2
			\end{bmatrix}^{\#_{\mathbb{A}}}\mathbb{V}\right)\\
 &=\frac{1}{2}  w_{\mathbb{A}}\left(\begin{bmatrix}
	T_1^{\#_A}-iT_2^{\#_A}&-i\sqrt{2}(T_1^{\#_A}+T_2^{\#_A})&2T_3^{\#_A}-iT_1^{\#_A}+T_2^{\#_A}\\
 i\sqrt{2}(T_2^{\#_A}-T_1^{\#_A})&2T_3^{\#_A}&\sqrt{2}(T_1^{\#_A}+T_2^{\#_A})\\
 2T_3^{\#_A}-(-iT_1^{\#_A}+T_2^{\#_A})&\sqrt{2}(T_2^{\#_A}-T_1^{\#_A})&T_1^{\#_A}-iT_2^{\#_A}
	\end{bmatrix}\right)\\
   &\leq w_{\mathbb{A}}\left(\begin{bmatrix}
	O&O&T_3\\
 O&T_3&O\\
 T_3&O&O
	\end{bmatrix}^{\#_{\mathbb{A}}}\right)+ \frac{1}{2}w_{\mathbb{A}}\left(\begin{bmatrix}
	T_1+iT_2&O&-(iT_1+T_2)\\
 O&O&O\\
 iT_1+T_2&O&T_1+iT_2
	\end{bmatrix}^{\#_{\mathbb{A}}}\right)\\
 &+\frac{1}{\sqrt{2}} w_{\mathbb{A}}\left(\begin{bmatrix}
	O&-i(T_2-T_1)&O\\
 i(T_1+T_2)&O&O\\
 O&O&O
	\end{bmatrix}^{\#_{\mathbb{A}}}\right)+ \frac{1}{\sqrt{2}}w_{\mathbb{A}}\left(\begin{bmatrix}
	O&O&O\\
 O&O&(T_2-T_1)\\
 O&T_1+T_2&O
	\end{bmatrix}^{\#_{\mathbb{A}}}\right)\\
 &= w_{\mathbb{A}}\left(\begin{bmatrix}
	O&O&T_3\\
 O&T_3&O\\
 T_3&O&O
	\end{bmatrix}\right)+ \frac{1}{2}w_{\mathbb{A}}\left(\begin{bmatrix}
	T_1+iT_2&O&-(iT_1+T_2)\\
 O&O&O\\
 iT_1+T_2&O&T_1+iT_2
	\end{bmatrix}\right)\\
 &+\frac{1}{\sqrt{2}} w_{\mathbb{A}}\left(\begin{bmatrix}
	O&-i(T_2-T_1)&O\\
 i(T_1+T_2)&O&O\\
 O&O&O
	\end{bmatrix}\right)+ \frac{1}{\sqrt{2}}w_{\mathbb{A}}\left(\begin{bmatrix}
	O&O&O\\
 O&O&(T_2-T_1)\\
 O&T_1+T_2&O
	\end{bmatrix}\right)\\
 &\leq  w_A(T_3)+\max\{ w_A(T_1),  w_A(T_2)\}+\frac{1}{\sqrt{2}}w_{\mathbb{A}}\left(\begin{bmatrix}
	O&-i(T_2-T_1)&O\\
 O&O&O\\
 O&O&O
	\end{bmatrix}\right)+\frac{1}{\sqrt{2}}w_{\mathbb{A}}\left(\begin{bmatrix}
	O&O&O\\
 i(T_1+T_2)&O&O\\
 O&O&O
	\end{bmatrix}\right)\\
 &+ \frac{1}{\sqrt{2}}w_{\mathbb{A}}\left(\begin{bmatrix}
	O&O&O\\
 O&O&(T_2-T_1)\\
 O&O&O
	\end{bmatrix}\right)+\frac{1}{\sqrt{2}}w_{\mathbb{A}}\left(\begin{bmatrix}
	O&O&O\\
 O&O&O\\
 O&T_1+T_2&O
	\end{bmatrix}\right)\\
 &=  w_A(T_3)+\max\{ w_A(T_1),  w_A(T_2)\}+ \frac{1}{\sqrt{2}}\left(\| T_1-T_2\|_A+\| T_1+T_2\|_A\right),
\end{align*}
 where the last equality follows by \eqref{eq1.5}.
 \end{proof}

 Again, using the arguments used in the proof of Proposition~\ref{p9}, we have the following.

\begin{cor}
    Let $T_i\in \mathcal{B}_A(\mathcal{H})$ for $1\leq i\leq 3$. Then
		\begin{align*}
	w_{\mathbb{A}}\left(\begin{bmatrix}
			T_1 & T_2& T_3\\
			T_2 & T_3&-iT_1\\
   T_3&-iT_1&-iT_2
			\end{bmatrix}\right)&\leq   w_A(T_3)+\max\{ w_A(T_1),  w_A(T_2)\}+ \frac{1}{\sqrt{2}}\left(\| T_1-T_2\|_A+\| T_1+T_2\|_A\right).
 \end{align*} 
\end{cor}

\begin{prop}\label{p10}
		Let $T_i\in \mathcal{B}_A(\mathcal{H})$ for $1\leq i\leq 4$. Then
		\begin{align*}
	w_{\mathbb{A}}\left(\begin{bmatrix}
			T_1 & T_2& T_3&T_4\\
			T_2 & T_3&T_4&iT_1\\
   T_3&T_4&iT_1&iT_2\\
   T_4&iT_1&iT_2&iT_3
			\end{bmatrix}\right)&\leq   w_A(T_4)+2\max\{ w_A(T_1),  w_A(T_3)\}\\
   &+\frac{1}{\sqrt{2}}\left(\| T_1-T_2\|_A+\| T_1+T_2\|_A\right)
 +\frac{1}{\sqrt{2}}\left(\| T_2-T_3\|_A+\| T_2+T_3\|_A\right).
 \end{align*} 
\end{prop}

 \begin{proof}
   Let $\mathbb{V}=\frac{1}{\sqrt{2}}\begin{bmatrix}
	iI&O&O & I\\
 O& \sqrt{2}I&O&O\\
 O&O& \sqrt{2}I&O\\
	I &O& O&iI
	\end{bmatrix}.$ 
It is not very difficult to show that $\mathbb{V}$ is $\mathbb{A}$-unitary. 
 Now, applying the arguments used in Proposition~\ref{p9}, using the fact that $w_A(U^{\#_A}TU)=w_A(T)$, and the relation $\mathcal{R}(T_i^{\#_A})\subseteq \overline{\mathcal{R}(A)}$, we have the following:
 \begin{align*}
   & w_{\mathbb{A}}\left(\begin{bmatrix}
			T_1 & T_2& T_3&T_4\\
			T_2 & T_3&T_4&iT_1\\
   T_3&T_4&iT_1&iT_2\\
   T_4&iT_1&iT_2&iT_3
			\end{bmatrix}\right)=     w_{\mathbb{A}}\left(\mathbb{V}^{\#_{\mathbb{A}}}\begin{bmatrix}
			T_1 & T_2& T_3&T_4\\
			T_2 & T_3&T_4&iT_1\\
   T_3&T_4&iT_1&iT_2\\
   T_4&iT_1&iT_2&iT_3
			\end{bmatrix}^{\#_{\mathbb{A}}}\mathbb{V}\right)\\
 &=\frac{1}{2}w_{\mathbb{A}}\left(\begin{bmatrix}
	T_1^{\#_A}-iT_3^{\#_A}&-i\sqrt{2}(T_1^{\#_A}+T_2^{\#_A})&-i\sqrt{2}(T_3^{\#_A}+T_2^{\#_A})&2T_4^{\#_A}-iT_1^{\#_A}+T_3^{\#_A}\\
 i\sqrt{2}(T_2^{\#_A}-T_1^{\#_A})&2T_3^{\#_A}&2T_4^{\#_A}&\sqrt{2}(T_1^{\#_A}+T_2^{\#_A})\\
 i\sqrt{2}(T_3^{\#_A}-T_2^{\#_A})&2T_4^{\#_A}&-i2T_1^{\#_A}&\sqrt{2}(T_2^{\#_A}+T_3^{\#_A})\\
 2T_4^{\#_A}-(-iT_1^{\#_A}+T_3^{\#_A})&\sqrt{2}(T_2^{\#_A}-T_1^{\#_A})&\sqrt{2}(T_3^{\#_A}-T_2^{\#_A})&T_1^{\#_A}-iT_3^{\#_A}
 	\end{bmatrix}\right)\\
   &\leq w_{\mathbb{A}}\left(\begin{bmatrix}
	O&O&O&T_4\\
 O&O&T_4&O\\
 O&T_4&O&O\\
 T_4&O&O&O
	\end{bmatrix}\right)+ \frac{1}{2}w_{\mathbb{A}}\left(\begin{bmatrix}
	T_1+iT_3&O&O&-(iT_1+T_3)\\
 O&O&O&O\\
 O&O&O&O\\
 iT_1+T_3&O&O&T_1+iT_3
	\end{bmatrix}\right)\\
  \end{align*}
  \begin{align*}
 &+w_{\mathbb{A}}\left(\begin{bmatrix}
 O&O&O&O\\
	O&T_3&O&O\\
 O&O&iT_1&O\\
 O&O&O&O
	\end{bmatrix}\right)+ \frac{1}{\sqrt{2}}w_{\mathbb{A}}\left(\begin{bmatrix}
	O&-i(T_2-T_1)&O&O\\
 i(T_1+T_2)&O&O&O\\
 O&O&O&O\\
 O&O&O&O
	\end{bmatrix}\right)\\
 &+\frac{1}{\sqrt{2}}w_{\mathbb{A}}\left(\begin{bmatrix}
	O&O&O&O\\
 O&O&O&(T_2-T_1)\\
 O&O&O&O\\
 O&T_1+T_2&O&O
	\end{bmatrix}\right)+ \frac{1}{\sqrt{2}}w_{\mathbb{A}}\left(\begin{bmatrix}
	O&O&-i(T_3-T_2)&O\\
  O&O&O&O\\
 i(T_3+T_2)&O&O&O\\
 O&O&O&O
	\end{bmatrix}\right)\\
 &+\frac{1}{\sqrt{2}} w_{\mathbb{A}}\left(\begin{bmatrix}
O&O&O&O\\
 O&O&O&O\\
 O&O&O&(T_3-T_2)\\
 O&O&(T_2+T_3)&O
 \end{bmatrix}\right)\\
 &\leq  w_A(T_4)+\max\{ w_A(T_1),  w_A(T_3)\}+\max\{ w_A(T_1),  w_A(T_3)\}\\
 &+\frac{1}{\sqrt{2}}w_{\mathbb{A}}\left(\begin{bmatrix}
	O&-i(T_2-T_1)&O&O\\
 O&O&O&O\\
 O&O&O&O\\
 O&O&O&O
	\end{bmatrix}\right)
 +\frac{1}{\sqrt{2}}w_{\mathbb{A}}\left(\begin{bmatrix}
	O&O&O&O\\
 i(T_1+T_2)&O&O&O\\
 O&O&O&O\\
 O&O&O&O
	\end{bmatrix}\right)\\
 &+ \frac{1}{\sqrt{2}}w_{\mathbb{A}}\left(\begin{bmatrix}
	O&O&O&O\\
 O&O&O&(T_2-T_1)\\
 O&O&O&O\\
 O&O&O&O
	\end{bmatrix}\right)+\frac{1}{\sqrt{2}}w_{\mathbb{A}}\left(\begin{bmatrix}
	O&O&O&O\\
 O&O&O&O\\
 O&O&O&O\\
 O&T_1+T_2&O&O
	\end{bmatrix}\right)\\
 &+\frac{1}{\sqrt{2}}w_{\mathbb{A}}\left(\begin{bmatrix}
	O&O&-i(T_3-T_2)&O\\
  O&O&O&O\\
 O&O&O&O\\
 O&O&O&O
	\end{bmatrix}\right)+\frac{1}{\sqrt{2}}w_{\mathbb{A}}\left(\begin{bmatrix}
	O&O&O&O\\
  O&O&O&O\\
 i(T_3+T_2)&O&O&O\\
 O&O&O&O
	\end{bmatrix}\right)
\end{align*}

\begin{align*}
    &+\frac{1}{\sqrt{2}} w_{\mathbb{A}}\left(\begin{bmatrix}
O&O&O&O\\
 O&O&O&O\\
 O&O&O&(T_3-T_2)\\
 O&O&O&O
 \end{bmatrix}\right)+\frac{1}{\sqrt{2}} w_{\mathbb{A}}\left(\begin{bmatrix}
O&O&O&O\\
 O&O&O&O\\
 O&O&O&O\\
 O&O&(T_2+T_3)&O
 \end{bmatrix}\right)\\
 &\leq  w_A(T_4)+2\max\{ w_A(T_1),  w_A(T_3)\}+\frac{1}{\sqrt{2}}\left(\| T_1-T_2\|_A+\| T_1+T_2\|_A\right)\\
 &+\frac{1}{\sqrt{2}}\left(\| T_2-T_3\|_A+\| T_2+T_3\|_A\right).
\end{align*}

 where the last equality follows by \eqref{eq1.5}.
 \end{proof}
As usual, it is not difficult to prove the following.
\begin{cor}
    Let $T_i\in \mathcal{B}_A(\mathcal{H})$ for $1\leq i\leq 4$. Then
		\begin{align*}
	w_{\mathbb{A}}\left(\begin{bmatrix}
			T_1 & T_2& T_3&T_4\\
			T_2 & T_3&T_4&-iT_1\\
   T_3&T_4&-iT_1&-iT_2\\
   T_4&-iT_1&-iT_2&-iT_3
			\end{bmatrix}\right)&\leq   w_A(T_4)+2\max\{ w_A(T_1),  w_A(T_3)\}\\
   &+\frac{1}{\sqrt{2}}\left(\| T_1-T_2\|_A+\| T_1+T_2\|_A\right)
 +\frac{1}{\sqrt{2}}\left(\| T_2-T_3\|_A+\| T_2+T_3\|_A\right).
 \end{align*}
\end{cor}

Motivated by \cite[Theorem~4.2]{AKR}, we have the following result for semi-Hilbert space.
\begin{lemma}\label{Lemma3.4}
    Let $S_{1k}, S_{k1}\in \mathcal{B}_A(\mathcal{H})$ for $1\leq k\leq n$. Then
		\begin{align*}
	w_{\mathbb{A}}\left(\begin{bmatrix}
	S_{11} & S_{12}& \dots&S_{1n}\\
	S_{21} & O& \dots&O\\
    \vdots& \vdots& \dots&\vdots\\
  S_{n1}& O& \dots&O
			\end{bmatrix}\right)&\leq \frac{1}{2}\left\{w_A(S_{11})+\sqrt{w_A^2(S_{11})+4\displaystyle\sum_{k=2}^n w_\mathbb{A}^2\left(\begin{bmatrix}
		O& S_{1k}\\ S_{k1} & O
		\end{bmatrix} \right)}\right\}.
\end{align*}
Further, if $S_{1k}=S_{k1}$, then
\begin{align*}
	w_{\mathbb{A}}\left(\begin{bmatrix}
	S_{11} & S_{12}& \dots&S_{1n}\\
	S_{12} & O& \dots&O\\
    \vdots& \vdots& \dots&\vdots\\
  S_{1n}& O& \dots&O
			\end{bmatrix}\right)&\leq \frac{1}{2}\left\{w_A(S_{11})+\sqrt{w_A^2(S_{11})+4\displaystyle\sum_{k=2}^n {w_A^2(S_{1k})}}\right\}.
\end{align*}
In particular, If $ S_{11}=O$, then
\begin{align*}
	w_{\mathbb{A}}\left(\begin{bmatrix}
	O & S_{12}& \dots&S_{1n}\\
	S_{12} & O& \dots&O\\
    \vdots& \vdots& \dots&\vdots\\
  S_{1n}& O& \dots&O
			\end{bmatrix}\right)&\leq \sqrt{\displaystyle\sum_{k=2}^n {w_A^2(S_{1k})}}.
\end{align*}
\end{lemma}
\begin{proof}
    Using Lemma \ref{l17}, we have
    \begin{align*}
w_{\mathbb{A}}\left(\begin{bmatrix}
	S_{11} & S_{12}& \dots&S_{1n}\\
	S_{21} & O& \dots&O\\
    \vdots& \vdots& \dots&\vdots\\
  S_{n1}& O& \dots&O
			\end{bmatrix}\right)&\leq  w_\mathbb{A}\left( \begin{bmatrix}
		w_A(S_{11}) & w_\mathbb{A}\left(\begin{bmatrix}
		O& S_{12}\\ S_{21} & O
		\end{bmatrix}
		\right) & \cdots & w_\mathbb{A}\left( \begin{bmatrix}
		O & S_{1n}\\ S_{n1} & O
		\end{bmatrix}\right) \\
		w_\mathbb{A}\left(\begin{bmatrix}
		O& S_{21}\\ S_{12} & O
		\end{bmatrix} \right) & O & \cdots & O\\
		\vdots & \vdots & \ddots & \vdots\\
		w_\mathbb{A}\left(\begin{bmatrix}
		O & S_{n1}\\ S_{1n} & O
		\end{bmatrix} \right) &  O & \cdots & O
		\end{bmatrix}\right).
   \end{align*}
   From Lemma \ref{lem0001} (ii), it is clear that $w_\mathbb{A}\left(\begin{bmatrix}
		O & S_{ij}\\ S_{ji} & O
		\end{bmatrix} \right)=w_\mathbb{A}\left(\begin{bmatrix}
		O & S_{ji}\\ S_{ij} & O
		\end{bmatrix} \right)$. Using this fact, it is very easy to check that the above matrix is $A$-self-adjoint. So, by using the identity  $w_\mathbb{A}(T)=r_\mathbb{A}(T)$, and $r_\mathbb{A}(TS)=r_\mathbb{A}(ST)$, we then have 
   \begin{align*}
&w_{\mathbb{A}}\left(\begin{bmatrix}
	S_{11} & S_{12}& \dots&S_{1n}\\
	S_{21} & O& \dots&O\\
    \vdots& \vdots& \dots&\vdots\\
  S_{n1}& O& \dots&O
			\end{bmatrix}\right)\\
   &\leq  r_\mathbb{A}\left( \begin{bmatrix}
		w_A(S_{11}) & w_\mathbb{A}\left(\begin{bmatrix}
		O& S_{12}\\ S_{21} & O
		\end{bmatrix}
		\right) & \cdots & w_\mathbb{A}\left( \begin{bmatrix}
		O & S_{1n}\\ S_{n1} & O
		\end{bmatrix}\right) \\
		w_\mathbb{A}\left(\begin{bmatrix}
		O& S_{12}\\ S_{21} & O
		\end{bmatrix} \right) & O & \cdots & O\\
		\vdots & \vdots & \ddots & \vdots\\
		w_\mathbb{A}\left(\begin{bmatrix}
		O & S_{1n}\\ S_{n1} & O
		\end{bmatrix} \right) &  O & \cdots & O
		\end{bmatrix}\right)\\
 &= r_\mathbb{A}\left( \begin{bmatrix}
		\frac{w_A(S_{11})}{2} & I & \cdots & O \\
		w_\mathbb{A}\left(\begin{bmatrix}
		O& S_{12}\\ S_{21} & O
		\end{bmatrix} \right) & O & \cdots & O\\
		\vdots & \vdots & \ddots & \vdots\\
		w_\mathbb{A}\left(\begin{bmatrix}
		O & S_{1n}\\ S_{n1} & O
		\end{bmatrix} \right) &  O & \cdots & O
		\end{bmatrix}\begin{bmatrix}
		I & O & \cdots &O \\
		\frac{w_A(S_{11})}{2} & w_\mathbb{A}\left(\begin{bmatrix}
		O& S_{12}\\ S_{21} & O
		\end{bmatrix} \right) & \cdots & w_\mathbb{A}\left(\begin{bmatrix}
		O & S_{1n}\\ S_{n1} & O
		\end{bmatrix} \right)\\
		\vdots & \vdots & \ddots & \vdots\\
		O&  O & \cdots & O
		\end{bmatrix}\right)\\
   &= r_\mathbb{A}\left( \begin{bmatrix}
		I & O & \cdots &O \\
		\frac{w_A(S_{11})}{2} & w_\mathbb{A}\left(\begin{bmatrix}
		O& S_{12}\\ S_{21} & O
		\end{bmatrix} \right) & \cdots & w_\mathbb{A}\left(\begin{bmatrix}
		O & S_{1n}\\ S_{n1} & O
		\end{bmatrix} \right)\\
		\vdots & \vdots & \ddots & \vdots\\
		O&  O & \cdots & O
		\end{bmatrix}\begin{bmatrix}
		\frac{w_A(S_{11})}{2} & I & \cdots & O \\
		w_\mathbb{A}\left(\begin{bmatrix}
		O& S_{12}\\ S_{21} & O
		\end{bmatrix} \right) & O & \cdots & O\\
		\vdots & \vdots & \ddots & \vdots\\
		w_\mathbb{A}\left(\begin{bmatrix}
		O & S_{1n}\\ S_{n1} & O
		\end{bmatrix} \right) &  O & \cdots & O
		\end{bmatrix}\right)\\
  &= r_\mathbb{A}\left( \begin{bmatrix}
		\frac{w_A(S_{11})}{2} & I & \cdots &O \\
		\frac{w_A^2(S_{11})}{4}+\displaystyle\sum_{k=2}^n w_\mathbb{A}^2\left(\begin{bmatrix}
		O& S_{1k}\\ S_{k1} & O
		\end{bmatrix} \right)& \frac{w_A(S_{11})}{2} & \cdots & O\\
		\vdots & \vdots & \ddots & \vdots\\
		O&  O & \cdots & O
		\end{bmatrix}\right)\\
  \end{align*}
  \begin{align*}
  &=\frac{1}{2}\left\{w_A(S_{11})+\sqrt{w_A^2(S_{11})+4\displaystyle\sum_{k=2}^n w_\mathbb{A}^2\left(\begin{bmatrix}
		O& S_{1k}\\ S_{k1} & O
		\end{bmatrix} \right)}\right\}.
   \end{align*}
\end{proof}
\begin{theorem}\label{Thm3.5}
     Let $S_{ij}$ be an $n\times n$ operator matrix with  $S_{ij}\in \mathcal{B}_A(\mathcal{H})$ for $1\leq i, j\leq n$. Then
		\begin{align*}
	w_{\mathbb{A}}\left(\begin{bmatrix}
	S_{11} & S_{12}& \dots&S_{1n}\\
	S_{21} & S_{22}& \dots&S_{2n}\\
    \vdots& \vdots& \dots&\vdots\\
  S_{n1}& S_{n2}& \dots&S_{nn}
			\end{bmatrix}\right)&\leq \frac{1}{2}\sum_{i=1}^n\left\{w_A(S_{ii})+\sqrt{w_A^2(S_{ii})+4\displaystyle\sum_{j=i+1}^n w_\mathbb{A}^2\left(\begin{bmatrix}
		O& S_{ij}\\ S_{ji} & O
		\end{bmatrix} \right)}\right\}.
\end{align*}
\end{theorem}
\begin{cor}\label{cor3.11} 
     Let $S_{ij}$ be an $n\times n$ operator matrix with  $S_{ij}\in \mathcal{B}_A(\mathcal{H})$ for $1\leq i, j\leq n$. Then
     \begin{enumerate}
         \item[(i)] \begin{align*}
	w_{\mathbb{A}}\left(\begin{bmatrix}
	S_{11} & S_{12}& \dots&S_{1n}\\
	S_{12} & S_{22}& \dots&S_{2n}\\
    \vdots& \vdots& \dots&\vdots\\
  S_{1n}& S_{2n}& \dots&S_{nn}
			\end{bmatrix}\right)&\leq \frac{1}{2}\sum_{i=1}^n\left\{w_A(S_{ii})+\sqrt{w_A^2(S_{ii})+4\displaystyle\sum_{j=i+1}^n w_\mathbb{A}^2\left( S_{ij}\right)}\right\}.
\end{align*}
  \item[(ii)] \begin{align*}
	w_{\mathbb{A}}\left(\begin{bmatrix}
	S_{11} & S_{12}& \dots&S_{1n}\\
	O & S_{22}& \dots&S_{2n}\\
    \vdots& \vdots& \dots&\vdots\\
 O& O& \dots&S_{nn}
			\end{bmatrix}\right)&\leq \frac{1}{2}\sum_{i=1}^n\left\{w_A(S_{ii})+\sqrt{w_A^2(S_{ii})+\displaystyle\sum_{j=i+1}^n  \|S_{ij}\|_A^2}\right\}.
\end{align*}
     \end{enumerate}	
\end{cor}
\begin{remark}\label{Remark3.6}
 Let $S_{1k}\in \mathcal{B}_A(\mathcal{H})$ for $1\leq k\leq n$. Then
\begin{align*}
	w_{\mathbb{A}}\left(\begin{bmatrix}
	S_{11} & S_{12}& \dots&S_{1n}\\
	O & O& \dots&O\\
    \vdots& \vdots& \dots&\vdots\\
  O& O& \dots&O
			\end{bmatrix}\right)&\leq \frac{1}{2}\left\{w_A(S_{11})+\sqrt{w_A^2(S_{11})+4\displaystyle\sum_{k=2}^n w_\mathbb{A}^2\left(\begin{bmatrix}
		O& S_{1k}\\ O & O
		\end{bmatrix} \right)}\right\}\\
  &=\frac{1}{2}\left\{w_A(S_{11})+\sqrt{w_A^2(S_{11})+\displaystyle\sum_{k=2}^n \|S_{1k}\|_A^2}\right\}.
\end{align*}
\end{remark}
\begin{remark}
\begin{enumerate}
    \item [(i)] \begin{align*}
	w_{\mathbb{A}}\left(\begin{bmatrix}
	T & S& \dots&S\\
	S & T& \dots&S\\
    \vdots& \vdots& \dots&\vdots\\
  S& S& \dots&T
			\end{bmatrix}\right)&\leq \frac{1}{2}\sum_{i=1}^n\left\{w_A(T)+\sqrt{w_A^2(T)+4\displaystyle\sum_{j=i+1}^n w_A^2(S)}\right\}.
  \end{align*}
      \item [(ii)] \begin{align*}
	w_{\mathbb{A}}\left(\begin{bmatrix}
	T & S& \dots&S\\
	O & T& \dots&S\\
    \vdots& \vdots& \dots&\vdots\\
 O& O& \dots&T
			\end{bmatrix}\right)&\leq \frac{1}{2}\sum_{i=1}^n\left\{w_A(T)+\sqrt{w_A^2(T)+\displaystyle\sum_{j=i+1}^n \|S}\|_A^2\right\}.
  \end{align*}
\end{enumerate}
\end{remark}

\begin{example}
    Let $V$ be a skew-symmetric Volterra operator acting on $ L^2 (-1,1)$ defined by 
  \begin{align*}
      (Vg)(t)=\int_{-t}^t g(x)dx. 
  \end{align*}   
 One can check that the values of the operator $V$ are contained in the set of all
odd functions from the space $ L^2 (-1,1)$ and for the odd function $g$, $Vg=0$. So,  $V^2=0$ i.e., $V$ is a nilpotent operator of index $2$. So, $w(V)=\frac{\|V\|}{2}=\frac{2}{\pi}$, as $\|V\|=\frac{4}{\pi}$ see \cite{PRHalmos1982}. 
     Let $S=[S_{ij}]$, for $1\leq i, j\leq n$, be an $n\times n$ operator matrix with each entries are Volterra operator mentioned above. Then, by using $\mathbb{A}=I$ in Corollary \ref{cor3.11}, we have
      \begin{align*}
	w(S)&\leq \frac{1}{2}\sum_{i=1}^n\left\{\frac{2}{\pi}+\sqrt{\frac{4}{\pi^2}+(n-i) \frac{16}{\pi^2}}\right\}.
\end{align*}
\end{example}

For the $n=5$ case, we need the following setup.
Using Lemma~\ref{l17} and Lemma~\ref{Lemma3.4}, we have the following.
 \begin{lemma}\label{Lemma4.8}
   Let $T_i\in \mathcal{B}_A(\mathcal{H})$ for $1\leq i\leq 2$. Then $$w_{\mathbb{A}}\left(\begin{bmatrix}
         O&T_1&O\\
         T_1&O&T_2\\
         O&T_2&O
     \end{bmatrix}\right)\leq \sqrt{w_A^2(T_1)+w_A^2(T_2)}.$$
 \end{lemma}
 \begin{proof}
     Let $\mathbb{U}=\begin{bmatrix}
         O&I&O\\
         I&O&O\\
         O&O&I
     \end{bmatrix}$ . One can check that $\mathbb{U}^{\#\mathbb{A}}\begin{bmatrix}
         O&T_1&O\\
         T_1&O&T_2\\
         O&T_2&O
     \end{bmatrix}^{\#\mathbb{A}}\mathbb{U}=\begin{bmatrix}
         O&T_1^{\#A}&T_2^{\# A}\\
         T_1^{\# A}&O&O\\
         T_2^{\# A}&O&O
      \end{bmatrix}$. Hence, the result follows from $w_A(U^{\#_A}TU)=w_A(T)$ and Lemma~\ref{Lemma3.4}. 
 \end{proof}

\begin{prop}\label{p11}
		Let $T_i\in \mathcal{B}_A(\mathcal{H})$ for $1\leq i\leq 5$. Then
		\begin{align*}
	w_{\mathbb{A}}\left(\begin{bmatrix}
			T_1 & T_2& T_3&T_4&T_5\\
			T_2 & T_3&T_4&T_5&iT_1\\
   T_3&T_4&T_5&iT_1&iT_2\\
   T_4&T_5&iT_1&iT_2&iT_3\\
   T_5&iT_1&iT_2&iT_3&iT_4
			\end{bmatrix}\right)&\leq   w_A(T_5)+\max\{ w_A(T_2),  w_A(T_3)\}+\max\{ w_A(T_1),  w_A(T_4)\}\\
   &+\frac{1}{\sqrt{2}}\sum_{i=1}^{3}\left(\| T_{i+1}-T_i\|_A+\| T_{i+1}+T_i\|_A\right)
 +\sqrt{w_A^2(T_4)+w_A^2(T_2)}.
 \end{align*} 
\end{prop}
\begin{proof}
Let $\mathbb{V}=\frac{1}{\sqrt{2}}\begin{bmatrix}
	iI&O&O &O& I\\
 O& \sqrt{2}I&O&O&O\\
 O&O& \sqrt{2}I&O&O\\
 O&O&O& \sqrt{2}I&O\\
	I &O&O& O&iI
	\end{bmatrix}$. It is not very difficult to show that $\mathbb{V}$ is $\mathbb{A}$-unitary. 
 Now, applying the arguments used in Proposition~\ref{p8} and Proposition~\ref{p9}, the fact that $w_A(U^{\#_A}TU)=w_A(T),$ and the relation $\mathcal{R}(T_i^{\#_A})\subseteq \overline{\mathcal{R}(A)}$, we have the following:
 \begin{align*}
     &w_{\mathbb{A}}\left(\begin{bmatrix}
			T_1 & T_2& T_3&T_4&T_5\\
			T_2 & T_3&T_4&T_5&iT_1\\
   T_3&T_4&T_5&iT_1&iT_2\\
   T_4&T_5&iT_1&iT_2&iT_3\\
   T_5&iT_1&iT_2&iT_3&iT_4
			\end{bmatrix}\right)=w_{\mathbb{A}}\left(\mathbb{V}^{\#_{\mathbb{A}}}\begin{bmatrix}
			T_1 & T_2& T_3&T_4&T_5\\
			T_2 & T_3&T_4&T_5&iT_1\\
   T_3&T_4&T_5&iT_1&iT_2\\
   T_4&T_5&iT_1&iT_2&iT_3\\
   T_5&iT_1&iT_2&iT_3&iT_4
			\end{bmatrix}^{\#_{\mathbb{A}}}\mathbb{V}\right)\\
 &=\frac{1}{2}w_{\mathbb{A}}\left(\scriptsize{\begin{bmatrix}
	T_1^{\#_A}-iT_4^{\#_A}&-i\sqrt{2}(T_1^{\#_A}+T_2^{\#_A})&-i\sqrt{2}(T_3^{\#_A}+T_2^{\#_A})&i\sqrt{2}(T_4^{\#_A}+T_3^{\#_A})&2T_5^{\#_A}-iT_1^{\#_A}+T_4^{\#_A}\\
 i\sqrt{2}(T_2^{\#_A}-T_1^{\#_A})&2T_3^{\#_A}&2T_4^{\#_A}&2T_5^{\#_A}&\sqrt{2}(T_1^{\#_A}+T_2^{\#_A})\\
 i\sqrt{2}(T_3^{\#_A}-T_2^{\#_A})&2T_4^{\#_A}&2T_5^{\#_A}&-i2T_1^{\#_A}&\sqrt{2}(T_2^{\#_A}+T_3^{\#_A})\\
 i\sqrt{2}(T_4^{\#_A}-T_3^{\#_A})&2T_5^{\#_A}&-i2T_1^{\#_A}&-i2T_2^{\#_A}&\sqrt{2}(T_4^{\#_A}+T_3^{\#_A})\\
 2T_5^{\#_A}-(-iT_1^{\#_A}+T_4^{\#_A})&\sqrt{2}(T_2^{\#_A}-T_1^{\#_A})&\sqrt{2}(T_3^{\#_A}-T_2^{\#_A})&\sqrt{2}(T_4^{\#_A}-T_3^{\#_A})&T_1^{\#_A}-iT_4^{\#_A}
	\end{bmatrix}}\right)
   \\ &\leq   w_A(T_5)+\max\{ w_A(T_1),  w_A(T_4)\}
   +\frac{1}{\sqrt{2}}\sum_{i=1}^{3}\left(\| T_{i+1}-T_i\|_A+\| T_{i+1}+T_i\|_A\right)\\
   &+w_{\mathbb{A}}\left(\begin{bmatrix}
       O&O&O&O&O\\
       O&T_3&T_4&O&O\\
       O&T_4&O&-iT_1&O\\
       O&O&-iT_1&-iT_2&O\\
       O&O&O&O&O
   \end{bmatrix}\right)\\
   &\leq   w_A(T_5)+\max\{ w_A(T_1),  w_A(T_4)\}
   +\frac{1}{\sqrt{2}}\sum_{i=1}^{3}\left(\| T_{i+1}-T_i\|_A+\| T_{i+1}+T_i\|_A\right)\\
   &+w_{\mathbb{A}}\left(\begin{bmatrix}
       O&T_4&O\\
       T_4&O&-iT_1\\
       O&-iT_1&O\\
   \end{bmatrix}\right)+w_{\mathbb{A}}\left(\begin{bmatrix}
       T_3&O&O\\
       O&O&O\\
       O&O&-iT_2\\
   \end{bmatrix}\right)\\
   &\leq   w_A(T_5)+\max\{ w_A(T_1),  w_A(T_4)\}
   +\frac{1}{\sqrt{2}}\sum_{i=1}^{3}\left(\| T_{i+1}-T_i\|_A+\| T_{i+1}+T_i\|_A\right)\\
 &+\sqrt{w_A^2(T_4)+w_A^2(T_2)}+\max\{ w_A(T_2),  w_A(T_3)\},
 \end{align*} 
 where the last inequalities follows from Lemma \ref{Lemma4.8} and Lemma \ref{lem0001}.
 \end{proof}

\section{Applications}
\subsection{Bounds for the eigenvalues of Matrix polynomials}
In this section, we present the numerical radius of the Frobenius companion matrix to derive new upper bounds for the eigenvalues of a monic matrix polynomial.
Let $P(z)=Iz^n+S_nz^{n-1}+\cdots +S_2z+S_1$ be a monic polynomial of degree $n\geq 2$ with matrix coefficients $S_1, \dots,S_n\in \mathbb{M}_n(\mathbb{C})$.  The Frobenius companion matrix of $P(z)$ is the matrix 
			  $$ C(P)=\begin{bmatrix}
			-S_n & -S_{n-1}  & \cdots  &   -S_2 & -S_1\\
			I  & O & \cdots  & O & O\\
			O  & I & \cdots  & O & O\\
			\vdots & \vdots &\ddots & \vdots & \vdots \\
			O  & O  &\cdots&  I & O
			\end{bmatrix}. $$\\
			This matrix establishes a crucial link between matrix analysis and geometry of matrix polynomials. A scalar $\lambda$ is an eigenvalue of $P(z)$ if there exists a non-zero vector $x\in \mathbb{C}^n$ such that $P(\lambda)x=0$.
			Moreover, $\lambda$ is an eigenvalue of $P(z)$ if and only if $\lambda \in \sigma(C(P))$ (see \cite{MFujiiFKuboPJA1973}).
			Since the eigenvalues of $C(P)$ are contained in its numerical range, if $\lambda$ is an eigenvalue of $C(P)$, then  
			$$|\lambda|\leq \rho(C(P))\leq w(C(P))\leq \|C(P)\|.$$
An important application of numerical radius inequalities is to bound the zeros of complex polynomials using a suitable partition of the well-known Frobenius companion matrix. Based on numerical radius  estimations of the companion matrix, various methods have been developed to find different upper bounds for the zeros of polynomials.
		The polynomial eigenvalue problem (computing and locating the eigenvalues of matrix polynomials) is a crucial topic in computation mathematics. The polynomial eigenvalue problem has recently received much attention. The authors of \cite{NJHighamFTisseurLAA2003,LeCTDuTHBNguyenTDOAM2019,AMelmanLAMA2019,BaniKitMus,JaradatKittaneh,Kittaneh2022} obtained bounds for the eigenvalues of matrix polynomials that are based on various matrix inequalities. Bounds for the eigenvalues of matrix polynomials that are based on norm and numerical radius inequalities can be found in \cite{NJHighamFTisseurLAA2003}. In 2019, Le {\it et al.} \cite{LeCTDuTHBNguyenTDOAM2019} used the norms of the coefficients matrices of a matrix polynomial to establish some (upper and lower) bounds for the eigenvalues of matrix polynomials. In 2019, Melman \cite{AMelmanLAMA2019} showed how $l$-ifications, namely, lower order matrix polynomials with the same eigenvalues as a given matrix polynomial, can be used to produce eigenvalue bounds.    
	 Using Remark \ref{Remark3.6}, we estimate the numerical radius of $C(P)$ to derive new upper bounds for the eigenvalues of $P(z)$.
  \begin{example}\label{Exmpshift}
	Let $R_s$ be the (right) shift operator on $\mathbb{C}^n$ defined by
	$$R_s=\begin{bmatrix}
	0 & 0 & \cdots  &   0 & 0\\
	1  & 0 & \cdots  & 0 & 0\\
	0  & 1 & \cdots  & 0 & 0\\
	\vdots & \vdots &\ddots & \vdots & \vdots \\
	0  & 0  &\cdots&  1 & 0
	\end{bmatrix}.$$ It is not difficult to see \textnormal{\cite{KEGustafsonDKMRao1991}} that
	$w(R_s)=\cos\left(\frac{\pi}{n+1}\right).$
\end{example}
 \begin{theorem}\label{Thm5.3}
 We have 
 \begin{align}
     w(C(P))\leq \frac{1}{2}\biggl\{w(S_{n})+\sqrt{w^2(S_{n})+4 w_\mathbb{A}^2\left(\begin{bmatrix}
		O& S_{n-1}\\I & O
		\end{bmatrix} \right)+\displaystyle\sum_{k=1}^{n-2} \|S_{k}\|^2}\biggr\}+\cos \left(\frac{\pi}{n}\right).
 \end{align}
 \end{theorem}
 \begin{proof}
 Note that $C(P)$ can be written as $C(P)=S+R_s$, where \\$ S=\begin{bmatrix}
			-S_n & -S_{n-1}  & \cdots  &   -S_2 & -S_1\\
			I  & O & \cdots  & O & O\\
			O  & O & \cdots  & O & O\\
			\vdots & \vdots &\ddots & \vdots & \vdots \\
			O  & O  &\cdots&  O & O
			\end{bmatrix} $ and $ R_s=\begin{bmatrix}
			O & O  & \cdots  &  O\\
			I  & O & \ddots    & \vdots\\
			O  & I & \ddots & \vdots\\
			\vdots & ~~~~~\ddots & ~~\ddots & \vdots  \\
			O  & O  &\cdots  I & O
			\end{bmatrix} $ .\\
			Now, using $\mathbb{A}=I$ in Lemma \ref{Lemma3.4} and Example \ref{Exmpshift}, we have
			\begin{align*}
			    w(C(P))&=w(S+R_s)\\
			    &\leq w(S)+w(R_s)\\
			    &\leq \frac{1}{2}\biggl\{w(S_{n})+\sqrt{w^2(S_{n})+4 w^2\left(\begin{bmatrix}
		O& S_{n-1}\\I & O
		\end{bmatrix} \right)+\displaystyle\sum_{k=1}^{n-2} \|S_{k}\|^2}\biggr\}+\cos \left(\frac{\pi}{n}\right).
			\end{align*}  
 \end{proof}
 The following result is for a different partition.
\begin{theorem}\label{Thm5.4}
 We have 
 \begin{align}
     w(C(P))\leq  \frac{1}{2}\biggl(w(S_{n})+\sqrt{w^2(S_{n})+\displaystyle\sum_{k=1}^{n-1} \|S_{k}\|^2} \biggr)+\cos \left(\frac{\pi}{n+1}\right). 
 \end{align}
 \end{theorem}
 \begin{proof}
 Note that $C(P)$ can be written as $C(P)=S+R_s$, where \\$ S=\begin{bmatrix}
			-S_n & -S_{n-1}  & \cdots  &   -S_2 & -S_1\\
			O & O & \cdots  & O & O\\
			O  & O & \cdots  & O & O\\
			\vdots & \vdots &\ddots & \vdots & \vdots \\
			O  & O  &\cdots&  O & O
			\end{bmatrix} $ and $ R=\begin{bmatrix}
			O & O  & \cdots  &  O\\
			I  & O & \ddots    & \vdots\\
			O  & I & \ddots & \vdots\\
			\vdots & ~~~~~\ddots & ~~\ddots & \vdots  \\
			O  & O  &\cdots  I & O
			\end{bmatrix} $ .\\
			Now, using $\mathbb{A}=I$ in Remark \ref{Remark3.6} and Example \ref{Exmpshift}, we have
			\begin{align*}
			    w(C(P))&=w(S+R_s)\\
			    &\leq w(S)+w(R_s)\\
			    &\leq \frac{1}{2}\biggl(w(S_{n})+\sqrt{w^2(S_{n})+\displaystyle\sum_{k=1}^{n-1} \|S_{k}\|^2} \biggr)+\cos \left(\frac{\pi}{n+1}\right).
			\end{align*}  
 \end{proof}
 \begin{example}
 Let us consider the $3\times 3$ monic matrix polynomial $P(z)=Iz^3+S_3z^{2}+S_2z+S_1$, whose coefficient matrices are given by
 $S_3=\begin{pmatrix}
     0 & 0 &0\\
       0 & 0 &0\\
         0 & 0 &0\\
 \end{pmatrix}$, $S_2=\begin{pmatrix}
     2 & 1 &1\\
       1 &2 &1\\
         1 & 1 &2\\
 \end{pmatrix}$, $S_1=\begin{pmatrix}
     4 & 1 &1\\
       1 &4 &1\\
         1 & 1 &4\\
 \end{pmatrix}$. If $\lambda$ is any zero of the polynomial $P(z)$, then Theorem \ref{Thm5.3} gives  $|\lambda|\leq  \frac{1}{2}\biggl\{\sqrt{4 w^2\left(\begin{bmatrix}
		O& S_{2}\\I & O
		\end{bmatrix} \right)+36}\biggr\}+\cos \left(\frac{\pi}{3}\right)=4.405$, while Theorem \ref{Thm5.4} gives $|\lambda|\leq 4.312658 $.
 \end{example}
 \subsection{Bounds for Foguel operators}
An operator $T\in \mathcal{L(H)}$ is said to be \textit{polynomially
bounded}\index{polynomially
bounded} if there exists an $M\geq 1$ such that $\|p(T)\|\leq
M\sup \{|p(z)|: |z|=1\}$ for all polynomials $p$. An operator
$T\in \mathcal{L(H)}$ is similar to a contraction if there exists a
bounded invertible operator $L$ such that $\|LTL^{-1}\| \leq 1$.
Halmos \cite{PRHalmosBAMS1970} raised a question whether every {\it polynomially
bounded} operator is similar to a contraction. There is a number of
results dealing with sufficient conditions for a {\it polynomially
bounded} operator to be similar to a contraction. Foias and Williams
\cite{JFCarlsonDNClarkCFoiasJPWilliams1994} studied the Foguel operators of the form T(X)=$\begin{pmatrix}
S^* & X\\
O & S
\end{pmatrix}$ acting on $H^2 \oplus H^2$, where $S$ is the forward unilateral shift on the Hardy space $H^2(\mathbb{T})$ of the unit circle $\mathbb{T}$ and $X \in \mathcal{L}(H^2(\mathbb{T}))$. They conjectured that there exists a Hankel operator $H_g$ with symbol $g$ such that the operator $T(H_g)$ is {\it polynomially bounded} and is not similar to a contraction. Petrovic \cite{SPetrovicPAMS1996} then exhibited a relationship between the similarity of $T(H_g)$ and a contraction. Now, we present a numerical radius inequality for a \textit{Foias-Williams operator}\index{Foias-Williams operator}. Let $T(H_g)=\begin{pmatrix}
S^* & H_g \\
O & S
\end{pmatrix}$, where $g(e^{i\theta})=-i(\pi-\theta)e^{-i\theta},~0\leq\theta \leq 2\pi.$ Then, $\|H_g\|\leq \|g\|_\infty=\pi$. Using $\mathbb{A}=I$ in Corollary \ref{cor3.11} (ii) and by triangle inequality for numerical radius, we obtain
\begin{align*}
w\begin{pmatrix}
S^* & H_g \\
O & S
\end{pmatrix} 
 \leq \frac{3}{2}+\frac{1}{2}\sqrt{1+\pi^2}.
\end{align*}

\begin{rem}
    We should mention here that the bound mentioned above is finer than the bound obtained by \cite[Remark 2.11]{SND}.
\end{rem}

\subsection{Bounds for little Hankel operators on the Bergman space}

Let $\mathbb{D}=\{z\in\mathbb{C}: |z|<1\}$ be the open unit disk in the complex plane $\mathbb{C}$ and $dA(z)=\frac{1}{\pi}dxdy$ be the Lebesgue area measure on  $\mathbb{D}$, normalized so that the area measure of $\mathbb{D}$ is equal to $1$. Let $L_a^2(\mathbb{D})$ be the  Bergman  space of all analytic functions  on $\mathbb{D}$  that are also in $L^2(\mathbb{D}, dA)$. Let ${P}$ be the  orthogonal projection   from $L^2(\mathbb{D}, dA)$ onto $L_a^2(\mathbb{D})$. Let $K(z, w)$ be the function on $\mathbb{D} \times \mathbb{D}$ defined by $K(z, w)=\frac{1}{(1-z\bar{w})^2},$ for $z, w \in \mathbb{D}$. 
	The function  $K(z, w)$ is called the {\it Bergman kernel} \cite{ZK} of $\mathbb{D}$ or the reproducing kernel of $L_a^2(\mathbb{D}).$ For $a \in \mathbb{D},$ let $k_a(z)=\frac{K(z,a)}{\sqrt{K(a, a)}}=\frac{(1-|a|^2)}{(1-\overline{a}z)^2}.$ The function $k_a$ is called the {\it normalized reproducing kernel} for $L_a^2(\mathbb{D}).$ 
	It is clear that $||k_a||_2=1.$ Let $L^{\infty}(\mathbb{D}, dA)$ denote the Banach space of Lebesgue measurable functions $f$ on $\mathbb{D}$ with $\|f\|_{\infty}=\mbox{ess sup}\{|f(z)|:z\in \mathbb{D}\} < \infty$. The space of all bounded analytic functions on $\mathbb{D}$ will be denoted by $H^{\infty}(\mathbb{D})$.  For $\phi \in L^{\infty}(\mathbb{D})$, define the little Hankel operator $S_\phi$ from $L_a^2(\mathbb{D})$ into $L_a^2(\mathbb{D})$ by $S_\phi f=P(J(\phi f))$, where $J$ is the mapping from $L^2(\mathbb{D})$ into $L^2(\mathbb{D})$ defined by $J(h(z))=h(\bar{z})$.
\begin{example}\label{Exp5.5}
		Let $\Phi=\begin{bmatrix}
		\phi_1 & \phi_1\\
		z\phi_2 & z\phi_3
		\end{bmatrix} $, where $\phi_1\in \overline{H^\infty(\mathbb{D})}$  and  $\phi_2,\phi_3\in H^\infty(\mathbb{D})$.
		Let ${\bf S}_{\Phi}=\begin{bmatrix}
		S_{\phi_1} & S_{\phi_1}\\
		S_{z\phi_2} & S_{z\phi_3}
		\end{bmatrix}$, where $S_{\phi_1},S_{z\phi_i}$ $i=2,3$ are little Hankel operators on $L_a^2(\mathbb{D})$ with symbol $\phi_1$ and $z\phi_2,$ $z\phi_3.$ Now since $S_{z\psi}$ defined on $L_a^2(\mathbb{D})$ is equivalent to zero \cite{ZK} if $\psi \in H^\infty(\mathbb{D}),$  we obtain   ${\bf S}_{\Phi}=\begin{bmatrix}
		S_{\phi_1} & S_{\phi_1}\\
		O & O
		\end{bmatrix}$. Notice that ${\bf S}_{\Phi}\in \mathcal{B}(L_a^2(\mathbb{D})\bigoplus L_a^2(\mathbb{D})).$ Setting $\mathbb{A}=I$ in Remark \ref{Remark3.6}, it follows that
		\begin{align*}
		w({\bf S}_{\Phi})
		&\leq \frac{1}{2}\left[w(S_{\phi_1})+\sqrt{w^2(S_{\phi_1})+\|S_{\phi_1}\|^2} \right]\\
		&\leq \frac{1}{2}\left[w(S_{\phi_1})+\sqrt{\|S_{\phi_1}\|^2+\|S_{\phi_1}\|^2} \right]\\
		&\leq\frac{1}{2}\left[ \|S_{\phi_1}\|+\sqrt{2}\|S_{\phi_1}\|\right]\\
		&= \frac{1+\sqrt{2}}{2}\|S_{\phi_1}\|\\
		&\leq \frac{1+\sqrt{2}}{2}\|\phi_1\|_\infty .
		\end{align*}
	\end{example}	
 \begin{rem}
By using a similar argument as used in the Example \ref{Exp5.5}, we have 
\begin{align*}
		w({\bf S}_{\Phi})\leq \frac{1+\sqrt{n}}{2}\|\phi_1\|_\infty ,
		\end{align*}
where ${\Phi}$ and ${\bf S}_{\Phi}$ are $n\times n$ operator matrices similar to the matrices defined as Example \ref{Exp5.5}, and $S_{\phi_1},S_{z\phi_i}$ $i=2,3, \dots, n+1$ are little Hankel operators on $L_a^2(\mathbb{D})$ with symbols $\phi_1\in \overline{H^\infty(\mathbb{D})}$, $\phi_i\in H^\infty(\mathbb{D}), i=2,\dots,n+1$  and $z\phi_2,$ $z\phi_3, \dots, z\phi_{n+1}$, respectively.
\end{rem}

	\vspace{.6cm}
	\noindent
	{\small {\bf Acknowledgments.}\\
	The first and third author are thankful to NISER Bhubaneswar for providing necessary facilities to carry out this works and also expresses their gratitude to Prof. Anil Kumar Karn. 
	}
	
	\bibliographystyle{amsplain}

\begin{thebibliography}{10}
		
       
         
         \bibitem{AKR} Abdollahi, O., S. Karami, and J. Rooin, \textit{New unitary equivalences for some operator matrices with applications}, Linear and Multilinear Algebra (2022), 1--22.
  
		\bibitem{AUK}Audenaert, K., \textit{A norm compression inequality for block partitioned positive semidefinite matrices}, Linear Algebra Appl. {\bf 413} (2006), 155--176.
		
		\bibitem{ARIS2} Arias, M. L., Corach, G., Gonzalez, M. C., \textit{Partial isometries in semi-Hilbertian spaces}, Linear Algebra Appl. {\bf 428} (2008), 1460--1475.
		
		\bibitem{MBKS} { Bakherad, M., Shebrawi, K.}, \textit{Upper bounds for numerical radius inequalities involving off-diagonal operator matrices}, Ann. Funct. Anal. {\bf 9} (2018), 297--309.
		
		\bibitem{BaniKit}{Bani-Domi, Kittaneh, F.}, \textit{Norm equalities and inequalities for operator matrices}, Linear Algebra Appl.
		{\bf 429} (2008),  57--67.
  
	    \bibitem{BaniKitShat}{Bani-Domi, Kittaneh, F., Shatnawi, M.}, \textit{New norm equalities and inequalities for certain operator matrices}, Math. Inequal. Appl. {\bf 23} (2020),  1041--1050.
		
        \bibitem{BaniKitMus}  Bani-Domi, W., Kittaneh, F., Mustafa, R., \textit{Bounds for the eigenvalues of matrix polynomials with commuting coefficients}, Results Math.  {\bf 78} (2023), no. 3, Paper No. 78, 18 pp.
  
		\bibitem{BhatiaKit} Bhatia, R., Kittaneh, F., \textit{ Norm inequalities for partitioned operators and an application}, Math. Ann. {\bf 287} (1990), 719--726.
		
		\bibitem{BhatiaKLi} Bhatia, R., Kahan, W., Li, R., \textit{Pinchings and norms of block scaled triangular matrices}, Linear Multilinear Algebra {\bf 50} (2002), 15--21.
		
		\bibitem{Pintu1} Bhunia, P., Paul, K., \textit{Some improvements of numerical radius inequalities of operators and operator matrices,} Linear Multilinear Algebra {\bf 70} (2022), 1995--2013.

  
		\bibitem{Pintu2} Bhunia, P., Nayak, R. K., Paul, K., \textit{Refinements of A-numerical radius inequalities and its applications}, 	Adv. Oper. Theory {\bf 5} (2020), 1498–1511.
		
		\bibitem{pinfek} Bhunia, P., Feki, K., Paul, K., \textit{$A$-numerical radius orthogonality and parallelism of semi-Hilbertian space operators and their applications,} Bull. Iran. Math. Soc. {\bf 47} (2021), 435–457.
		
		\bibitem{PINTU} Bhunia, P., Paul, K., Nayak, R. K., \textit{On inequalities for A-numerical radius of operators}, Electron. J. Linear Algebra \textbf{36} (2020), 143--157.

        \bibitem{BoseHS} Bose, A., Hazra, R. S., Saha, K., \textit{Spectral norm of circulant-type matrices}, J. Theoret. Probab. {\bf 24} (2011), 479--516.
		
	
		\bibitem{JFCarlsonDNClarkCFoiasJPWilliams1994}Carlson, J. F., Clark, D. N., Foias, C. and Williams, J. P., \textit {Projective Hilbert $\bold A(\bold D)$-modules}, New York J. Math. \textbf{1} (1994), 26--38.
		
		\bibitem{Davis} Davis, P. J.,  \textit{Circulant Matrices}, Chelsea Publishing, New York, 1994.
		
		\bibitem{Doug} Douglas, R. G., \textit{On majorization, factorization, and range inclusion of operators on Hilbert space},
		Proc. Amer. Math. Soc. \textbf{17} (1966), 413--415.
		
		
		
		\bibitem{faiot}{Feki, K.}, {Generalized numerical radius inequalities of operators in Hilbert spaces,}  Adv. Oper. Theory \textbf{6} (2021), \url{https://doi.org/ 10.1007/s43036-020-00099-x}.
	


		\bibitem{feki03} {Feki, K.}, {A note on the $A$-numerical radius of operators in semi-Hilbert spaces}, Arch. Math. {\bf 115} (2020) 535--544 . 

  
		
		\bibitem{Feki} Feki, K., \textit{Some $\mathbb{A}$-numerical radius inequalities for $d\times d$ operator matrices,} Rendiconti del Circolo Matematico di Palermo Series 2 {\bf 71} (2022), 85--103.



  
		\bibitem{Feki01} Feki, K., \textit{Spectral radius of semi-Hilbertian space operators and its applications,} Ann.  Funct. Anal. {\bf 11} (2020), 929-946.

        \bibitem{MFujiiFKuboPJA1973}  Fujii, M. and Kubo, F., \textit{Operator norms as bounds for roots of algebraic equations}, Proc. Japan. Acad. {\bf 49} (1973), 805--808.

       \bibitem{KEGustafsonDKMRao1991} Gustafson, K. E. and Rao, D. K. M., \textit{Numerical range. The field of values of linear operators and matrices}, {Springer-Verlag}, New York, 1997.
  
		\bibitem{PRHalmos1982} {Halmos, P. R.}, \textit{A Hilbert space problem book}, Second edition, Graduate Texts in Mathematics, 19. Encyclopedia of Mathematics and its Applications, {\bf 17} Springer-Verlag, New York-Berlin, 1982.

        \bibitem{PRHalmosBAMS1970} Halmos, P. R., \textit{Ten problems in Hilbert space}, Bull. Amer. Math. Soc. \textbf{76} (1970), 887--933.

        \bibitem{NJHighamFTisseurLAA2003}  Higham, N. J. and Tisseur, F., \textit{ Bounds for eigenvalues of matrix polynomials}, Linear Algebra Appl. {\bf 358} (2003), 5--22.

        \bibitem{HirKit} Hirzallah, O., Kittaneh, F.,  Shebrawi, K., \textit{Numerical radius inequalities for  $2\times 2$ operator matrices}, Studia Math.  \textbf{210} (2012), 99--115. 

		\bibitem{TY} Hirzallah, O., Kittaneh, F.,  Shebrawi, K., \textit{Numerical radius inequalities for certain $2\times 2$ operator matrices}, Integr. Equ. Oper. Theory \textbf{71} (2011), 129--147.
  
        \bibitem{JaradatKittaneh}  Jaradat, A. Kittaneh, F., \textit{Bounds for the eigenvalues of monic matrix polynomials from numerical radius inequalities}, Adv. Oper. Theory {\bf 5} (2020), no. 3, 734–743.
  
		\bibitem{JXU} Jiang, Z. L., Xu, T. T.,  \textit{Norm estimates of $\omega$-circulant operator matrices and isomorphic operators for
			$\omega$-circulant algebra}, Sci. China Math. {\bf 59} (2016), 351--366.
		
			
		\bibitem{JZ} Jiang, Z. L., Zhou, Z. X., \textit{Circulant Matrices}, Chengdu Technology University Publishing Company, Chengdu, 1999.
		
	
		
		
		\bibitem{KC} King, C., \textit{Inequalities for trace norms of $2\times 2$ block matrices}, Comm. Math. Phys. {\bf 242} (2003), 531--545.
		
		\bibitem{KC1} King, C., Nathanson, M., \textit{New trace norm inequalities for $2\times 2$ blocks of diagonal matrices}, Linear Algebra Appl. {\bf 389} (2004), 77--93.
		
		\bibitem{Kissin} Kissin, E., \textit{On Clarkson–McCarthy inequalities for $n$-tuples of operators}, Proc. Amer. Math. Soc. {\bf 135} (2007), 2483--2495.
  
       \bibitem{Kittaneh2022}  Kittaneh, F., \textit{Spectral radius inequalities for operator matrices with commuting entries}, Proc. Amer. Math. Soc. {\bf 150} (6) (2022), 2593–2601.

        \bibitem{KITSAT} Kittaneh, F., Sahoo, \textit{On $ \mathbb{A}$-numerical radius equalities and inequalities for certain operator matrices}, Ann. Funct. Anal. \textbf{12} (52) (2021). \url{https://doi.org/10.1007/s43034-021-00137-6}

        \bibitem{LeCTDuTHBNguyenTDOAM2019} Le, C. T., Du, T. H. B. and Nguyen, T. D., \textit{On the location of eigenvalues of matrix polynomials}, Oper. Matrices {\bf 13} (4) (2019), 937--954.
			
		\bibitem{LIJL}Li, J., Jiang, Z. L, Lu, F. L., \textit{Determinants, norms, and the spread of circulant matrices with Tribonacci and Generalized Lucas numbers}, Abstr. Appl. Anal. 2014, Article ID 381829, 9 pages (2014).
		
		
		
		\bibitem{Manu}Manuilov, V., Moslehian, M. S., Xu, Q., \textit{Douglas factorization theorem revisited,} Proc. Amer. Math. Soc. \textbf{148} (2020), 1139--1151.
  
		\bibitem{AMelmanLAMA2019} Melman, A., \textit{Polynomial eigenvalue bounds from companion matrix polynomials}, Linear Multilinear Algebra {\bf 67} (2019), 598--612.

       \bibitem{Mos} Moslehian, M. S., Kian, M., Xu, Q.,  \textit{Positivity of $2\times 2$ block matrices of operators,} Banach J. Math. Anal. \textbf{13} (2019), 726--743.
  
		\bibitem{MOS_SAT} Moslehian, M. S., Sattari, M., \textit{Inequalities for operator space numerical radius of $2\times 2$ block matrices,} J. Math. Phys. \textbf{57} (2016),  015201, 15pp.
		
		\bibitem{MOS} Moslehian, M. S., Xu, Q., Zamani, A.,  \textit{Seminorm and numerical radius inequalities of operators in semi-Hilbertian spaces}, Linear Algebra Appl. \textbf{591} (2020), 299--321.
		
		\bibitem{Nashed}  Nashed, M. Z., \textit{Generalized Inverses and Applications}, Academic Press, New York, 1976.

        \bibitem{SPetrovicPAMS1996} {Petrovi\'c, S.,} \textit{Some remarks on the operator of Foias and Williams}, {Proc. Amer. Math. Soc.} {\bf 124} (9) (1996), 2807--2811.
		
		\bibitem{NSD} Rout, N. C., Sahoo, S., Mishra, D., \textit{Some $A$-numerical radius inequalities for semi-Hilbertian space operators}, Linear  Multilinear Algebra {\bf 69} (2021), 980--996.
		
		\bibitem{Nirmal2} Rout, N. C., Sahoo, S., Mishra, D.,
		\textit{On $\mathbb{A}$-numerical radius inequalities for $2\times2$ operator matrices}, Linear  Multilinear Algebra \textbf{70}(14) (2022), 2672--2692.
		
		
		
		\bibitem{Saddi} Saddi, A., \textit{A-normal operators in semi Hilbertian spaces,} AJMAA \textbf{9} (2012), 1--12. 
		
		\bibitem{SS}  Sahoo, S., \textit{On $\mathbb{A}$-numerical radius inequalities for $2\times2$ operator matrices-II}, Filomat {\textbf {35}} (15) (2021), 5237--5252.

        \bibitem{SND2} Sahoo, S.,  Das, N., Mishra, D., \textit{Berezin number and numerical radius inequalities for operators on Hilbert spaces}, Adv. Oper. Theory {\textbf {5}}(2020), 714–727.

  
		\bibitem{SND}  Sahoo, S.,  Das, N.,  Mishra, D.,
		\textit{Numerical radius inequalities for operator matrices}, Adv. Oper. Theory {\bf 4} (2019), 197--214.
		
		\bibitem{SND1}  Sahoo, S., Rout, N. C., Sababheh,  M., \textit{Some extended numerical radius inequalities}, Linear  Multilinear Algebra {\bf 69} (2021), 907--920.
 
		
			
		\bibitem{XYZ} Xu, Q., Ye, Z.,  Zamani, A., \textit{Some upper bounds for the $ \mathbb{A}$-numerical radius of $2\times2$ block matrices}, Adv. Oper. Theory \textbf{6} (2021) https://doi.org/10.1007/s43036-020-00102-5
		
		\bibitem{Zam} Zamani, A.,  \textit{A-Numerical radius inequalities for semi-Hilbertian space operators,} Linear Algebra Appl. \textbf{578} (2019), 159--183.

		\bibitem{ZK} Zhu, K., {Operator theory in function spaces.} Monographs and Textbooks in Pure and Applied Mathematics, {\bf 139}, Marcel Dekker, Inc., New York, 1990.

	\end{thebibliography}
	
\end{document}